\documentclass{amsart}
\usepackage[margin=2cm]{geometry} %Margins
\usepackage{setspace} % Line spacing
\usepackage{graphicx}%for inserting images
\usepackage[boxsize= .3 em]{ytableau}

\usepackage[]{hyperref} % makes coloured table of contents 
\hypersetup{colorlinks=true,} 
\usepackage{xcolor} % general coloured text

\usepackage{amsmath,amsthm,amssymb,amsxtra} 
\usepackage{mathtools}
\usepackage{xfrac} % for split-level quotients
\usepackage{quiver} % for commutative diagrams from quiver

\usepackage[capitalise]{cleveref}

\numberwithin{equation}{section}

\newtheorem{theorem}{Theorem}[section]
\newtheorem{corollary}[theorem]{Corollary}
\newtheorem{lemma}[theorem]{Lemma}

\newtheorem{proposition}[theorem]{Proposition}

\newtheorem{conjecture}[theorem]{Conjecture}

\theoremstyle{definition}
\newtheorem{definition}[theorem]{Definition}

\theoremstyle{remark}

\newtheorem{remark}[theorem]{Remark}

\newtheorem{example}{Example}[theorem]
\AtBeginEnvironment{example}{%
  \pushQED{\qed}%
}
\AtEndEnvironment{example}{\popQED\endexample}

\newcommand{\N}{\mathbb{N}}
\newcommand{\Z}{\mathbb{Z}}
\newcommand{\C}{\mathbb{C}}
\newcommand{\R}{\mathbb{R}}

\newcommand{\Fl}{\operatorname{Fl}}
\newcommand{\MaxDiag}{\operatorname{MaxDiag}}

\DeclareMathOperator{\minor}{Minor}
       
\newcommand{\qbinom}[2]{\left[ \genfrac{}{}{0pt}{}{#1}{#2} \right]}
\DeclareMathOperator{\Toep}{\mathrm{Toep}}

\newcommand{\inv}{^{-1}}

\newcommand{\dedication}[1]{%
  \begin{center}
    \textit{#1}
  \end{center}
}

\title{Grassmannian Quantum cohomology in the infinite limit and total positivity}
\author{In\'es Chung-Halpern and Konstanze Rietsch}

\thanks{In\'es Chung-Halpern is funded by a London School of Geometry and Number Theory–King’s College London PhD studentship, which is supported by the Engineering and Physical Sciences Research Council [EP/S021590/1]}

\begin{document}

\maketitle

\dedication{In honour of George Lusztig on the occasion of his $80^{th}$ birthday.}

\begin{abstract} 
The theory of total positivity  was shown by Lusztig to be intrinsically linked to the canonical basis with its positivity properties.  When we restrict ourselves to studying total positivity just for the set of lower-triangular unipotent Toeplitz matrices,  say in type $A$, then there is a similar link with the quantum cohomology rings of  flag varieties and the Schubert bases and their positivity properties. Namely, this builds on a theory of Dale Peterson \cite{peterson} that gives a uniform Lie-theoretic description of all of the quantum cohomology rings $qH^*(G/P)$.  
In a precursor \cite{rietsch2025totallypositivetoeplitzmatrices} to this paper, the Schubert basis and quantum parameters in $qH^*(SL_n/B)$, which restrict to positive-valued functions on totally positive Toeplitz matrices,  were analysed with respect to their limiting behaviour as $n\to\infty$, uncovering a novel connection with the classical Edrei theorem on parametrising the infinite totally positive Toeplitz matrices.  
In this paper we study the Grassmannian case, using the conventions from the $SL_{n}/B$ setting as a guide, and we determine the quantum parameter and Schubert class asymptotics in different scenarios. 
Along the way, we obtain a new interpretation of the strange duality involution on the localised quantum cohomolgy ring of the Grassmannian. Finally, we prove an asymptotic formula for quantum parameters in a partial flag setting, and we furthermore formulate some conjectures concerning partial flag varieties and related quantum cohomology asymptotics.  
 \end{abstract}

 \section{Introduction}
 \label{sec:Intro}
Classically, a totally positive matrix is a real matrix all of whose minors are positive. The theory of totally positive matrices dates back to the early 20$^{\rm th}$ century pioneered  by mathematicians including Polya, Fekete, Gantmacher, Schoenberg and others. Lusztig's work on total positivity, beginning in the 1990's with the seminal paper \cite{Lusztig94}, gave a generalisation of the theory to all reductive algebraic groups and their homogeneous spaces, as well as uncovering two relationships with his canonical basis: one, the fact that in simply-laced types totally positive elements act by positive matrices in \textit{any} irreducible representation with regard to the canonical basis of that representation, the other, a remarkable interpretation of the parametrisation of canonical basis elements using a tropical version of total positivity. Lusztig's body of work on total positivity, including~\cite{Lusztig94,LUSZTIG1997,Lusztig1998TotalPI,Lusztig2019TotalPI,LUSZTIG2023}, has let to a resurgence of this field and has had a very wide range of further impacts. We mention here Fomin and Zelevinski's theory of cluster algebras~\cite{FOMINICM2010}, Postnikov's combinatorial perspective on the totally positive Grassmannian~\cite{postnikovgrassmannian}, generalisations to Kac-Moody groups~\cite{LamPP,BaoHe}, applications to the study of moduli of local systems on a surface \cite{FockGoncharovI}, integrable systems (KP equation) related to water waves \cite{KodamaWilliams}, and calculation of scattering amplitudes~\cite{Arkani-Hamedetalbook,ArkaniHamed2013jha} as examples.

Relevant to our purposes is another application,  which connects the theory of total positivity with the theory of quantum cohomology for flag varieties $G/P$, Toda lattices and Toeplitz matrices~\cite{peterson,rietsch2001flagvarieties}. Our goal will be to study the typa $A_n$ asymptotics of this theory with particular emphasis on the Grassmannian case. We now focus on type $A$ and
 begin by recalling a theorem that relates totally nonnegative Toeplitz matrices to quantum cohomology rings of partial flag varieties $SL_{n+1}/P$. This theorem was  inspired by Lusztig's characterisation of the totally positive part of $G$ as consisting of elements $g\in G$ whose canonical basis matrix coefficients are positive for all irreducible representations of $G$. To state the theorem, let us first define Dale Peterson's `Toeplitz variety' associated to $SL_{n}/P$. 
 We use standard notations from algebraic groups for which a reference is \cite{Springer:book}.  Let $B$ be the upper-triangular Borel subgroup of $SL_n$ (also denoted $B_n$), and let  $P\supseteq B$ be the parabolic subgroup associated to the partial flag variety $SL_n/P=\Fl_{n_1,\dotsc,n_k}(\C^n)$ of flags of subspaces $V_i$ with $\dim(V_i)=n_i$.  We may also write $P_{\mathbf n}$ for $P$ if we wish to keep track of the dimension vector $\mathbf n=(n_1,\dotsc, n_k)$.
 We  write $I=\{1,\dotsc, n-1\}$ for the indexing set of simple reflections $\{s_i\mid i\in I\}$ of the Weyl group $W=S_n$, and $I^P$ for the subset $\{n_1,\dotsc, n_k\}$ corresponding to $P$. Let $W_P$ be the parabolic subgroup of $W$ associated to $P$. The Schubert cells in $SL_n/P$ are indexed by the set of minimal coset representatives $W^P$ for $W/W_P$, which is also the set of $w\in W$ for which all reduced expressions end in $s_i$ with $i\in I^P$.  In the following, all varieties will be defined over $\C$, and may be identified with their $\C$-valued points. Let $U_-\subset SL_n(\C)$ be the unipotent lower-triangular subgroup, and  write $U_-(\R_{\ge 0})$ for the totally nonnegative part of $U_-$.  
\begin{definition}\label{d:PetersonToeplitz}
Let $X\subset U_-$ be the subgroup of Toeplitz matrices (matrices with constant entries along diagonals). This is also the stabilizer subgroup of the standard principal nilpotent $F=\sum F_i$, see Section~\ref{s:conventions} for more detail. The (type $A$) Peterson Toeplitz variety associated to the parabolic $P=P_{\mathbf n}$ is defined by
\[
X_P=X\cap B_+w_0w_P B_+,
\]
where $w_0w_P$  is the longest element in $W^P$. Let $X_P(\R_{>0})$ denote the `totally positive part' which is the intersection of $X_P$ with the set of totally nonnegative matrices in $U_-$. \end{definition}
Let $\Delta_j(u)$ denote the $j\times j$ lower left-hand corner minor of a matrix $u$. We have the following concrete description of the $\C$-valued points of $X_P$, where we also use the notation $\mathbf c=(c_{1},\dots, c_{n-n_1})$, 
\begin{equation}\label{e:XPCgen}
X_P(\C):=\left\{\left .u_{n}(\mathbf c)=\begin{pmatrix}
 1 &  & &  &        & \\
  c_{1} &1   &    &       & & \\
  \vdots &c_{1} &  \ddots &  &      &  \\
   c_{n-n_1}&     &      \ddots  & \ddots &   &  \\
   &   \ddots  &       &    c_{1}  &  1    &  \\
   0&     &     c_{n-n_1}  &     \cdots  &   c_{1}    &1
\end{pmatrix}\right | \Delta_{j}(u_n(\mathbf c))\ne 0 \text{ if and only if $j\in \{n_1,\dotsc, n_k\}$}\right\}.
\end{equation}
Note that the would-be lower right-hand corner entries $ c_{n-1},\dotsc, c_{n-n_1+1}$ are automatically $0$ as a consequence of the condition that $\Delta_j(u_n(\mathbf c))=0$ for $j=1,\dotsc, n_1-1$. The totally positive part $X_P(\R_{>0})$ is the set of those $u_{n}(c_{1},\dots c_{n-n_1})\in X_P$ for which all nontrivial minors (the entries $c_i$ included) are in $\R_{>0}$.

Recall that the (small) quantum cohomology ring of $X=SL_{n}/P$ is a deformation of the usual cohomology with deformation parameters $q_{n_i}$, called `quantum parameters'. These quantum parameters correspond to generators of the Picard group $\operatorname{Pic}(X)=\langle\mathcal L_{\omega_{n_i}}\rangle_{i=1,\dotsc, k}$, which are represented by the line bundles associated to the fundamental weights $\omega_{n_i}$ with $n_i\in I^P$. Thus 
\[qH^*(SL_{n}/P,\Z)=H^*(SL_{n}/P,\Z)\otimes_\Z \Z[q_{n_1},\dotsc, q_{n_k}],
\]
as a module over $\Z[q_{n_1},\dotsc, q_{n_k}]
$. The structure constants in terms of  the Schubert basis $\{\sigma_P^w\mid w\in W^P\}$ are  given by positive polynomials in the $q_{n_j}$, each coefficient having an interpretation as an enumerative genus~$0$, $3$-point Gromov-Witten invariant. The quantum cup product $\star$ recovers the standard cup product $\cup$, when setting the quantum parameters to $0$. 
By Dale Peterson's theory, we have the following description of the quantum cohomology ring of $SL_n/P$ after inverting quantum parameters, which is in terms of the Peterson Toeplitz variety $X_P$. Our chosen conventions here differ from those in \cite{rietsch2001flagvarieties} in a way that we explain in \cref{s:conventions}.
\begin{theorem}[Dale Peterson~\cite{peterson,rietsch2001flagvarieties,rietsch2001err}]\label{t:Peterson} Let $P=P_{\mathbf n}$ in $SL_n$. To any Grassmannian permutation $w$ with descent $j$   let $\Delta^w_j$ be its associated  $j\times j$ minor, see Definition~\ref{d:wminornew}. There is an isomorphism
\begin{eqnarray*}
\Phi_P:qH^*(SL_{n}/P,\C)[q_{n_1}\inv,\dotsc, q_{n_k}\inv]&\overset\sim\longrightarrow &\C[X_P]
\end{eqnarray*}
which is determined by the formula 
\begin{equation}\label{e:PhiP}\Phi_P(\sigma_P^{w})=\frac{\Delta_{n_i}^w}{\Delta_{n_i}},
\end{equation} 
where $w$ is a Grassmannian permutations with descent $n_i\in I^P$. 
\end{theorem}

Let us write $\mathfrak S^w_{P}:=\Phi_P(\sigma^w_P)$ and keep the notation $q_{n_j}$ for $\Phi_P(q_{n_j})$. Thus, if $w$ is Grassmannian permutation then $\mathfrak S^w_P$ is as given in \eqref{e:PhiP}. This `Peterson Toeplitz isomorphim' endows the coordinate ring $\C[X_P]$ of $X_P$ with a remarkable basis,
\begin{equation}\label{e:basis}
\left\{\left. \big(\prod_{i=1}^kq^{m_i}_{n_i}\big)\mathfrak S^{w}_P \right| (m_i)_{i=1}^k\in\Z^{k}, w\in W^P\right\}.
\end{equation}
This basis has positive integer structure constants, reminding very much of the theory of total positivity and the role played by Lusztig's canonical basis (via its associated matrix coefficients). Indeed, the relationship is more than close, since the elements of $X_P(\C)$ where all the Schubert basis elements take nonnegative values turns out to agree precisely with what we will refer to as the totally positive part of $X_P$,
\[
X_P(\R_{>0}):=X_{P}(\C)\cap U_-(\R_{\ge 0}).
\] 
We also note that the positivity properties of the Schubert basis are very central to the structure of these quantum cohomology rings.     
In fact, in the Grassmannian case (and conjecturally in general), the Schubert basis is determined by its positivity properties, see \cite{BuchWang}. 

 We now recall the parametrisation theorem for $X_P(\R_{>0})$ that was proved in \cite{rietsch2001flagvarieties} using Peterson's isomorphism. A generalisation to other types was given in \cite{LamRietsch} (however with the additional condition of positivity of quantum parameters needed for  \eqref{e:poscharToep}). A different proof of the final part of this theorem using ideas from mirror symmetry was given in \cite{rietschNagoya}.
 \begin{theorem}[{\cite{rietsch2001flagvarieties}}] \label{t:totposGP}Interpreting the elements of $qH^*(SL_{n}/P)$ as functions on $X_P$ via Theorem~\ref{t:Peterson}, we have that $\sigma_P^w$ and the quantum parameters $q_{n_i}$ take positive values on the totally positive part $X_P(\R_{>0})$. Moreover, we have the following characterisation of $X_P(\R_{>0})$, 
 \begin{equation}\label{e:poscharToep}
 X_P(\R_{>0})=\{u\in X_P\mid \sigma^w_P(u)>0\ \text{ for all $w\in W^P$} \},
 \end{equation}
and the map $ \Delta_{\mathbf n}:X_P(\R_{>0})\to\R_{>0}^k$ given by 
 \begin{equation}\label{e:poschartToep}
u=u_n(c_1,\dotsc, c_{n-n_k})\mapsto (\Delta_1(u),\dotsc,\Delta_k(u))
 \end{equation}
 is a homeomorphism.
 \end{theorem}
 We note that the inverse map to \eqref{e:poschartToep}, recovering $u$ from its `parameters' $\Delta_1(u),\dotsc,\Delta_k(u)$, is not algebraic. It involves solving a set of polynomial equations that cannot be explicitly solved.
\begin{remark}\label{r:diforfullflag}  In the full flag variety case another useful version of the above parametrisation map describes it as a homeomorphism from $X_B(\R_{>0})$ to $T_{SL_n}(\R_{>0})$, the positive part of the maximal torus of $SL_n$. Namely,
\begin{equation}
    u\mapsto \begin{pmatrix}d_1(u)& & & &\\
    &d_2(u) & & &\\
    & &\ddots & &\\
    & & & d_{n-1}(u)&\\
    & & & &d_n(u)
    \end{pmatrix}:=\begin{pmatrix}\Delta_1(u)& & & &\\
    &\frac{\Delta_2(u)}{\Delta_1(u)} & & &\\
    & &\ddots & &\\
    & & & \frac{\Delta_{n-1}(u)}{\Delta_{n-2}(u)}&\\
    & & & &\frac{1}{\Delta_{n-1}(u)}
    \end{pmatrix}.
\end{equation}
See also \cite[Theorem~2.1]{rietsch2025totallypositivetoeplitzmatrices}.
\end{remark}
The interest in the $n\to\infty$ asymptotics of Theorem~\ref{t:totposGP} arose from a beautiful classical theory of total positivity for infinite  Toeplitz matrices dating back to the early 1950's. Namely, in the infinite setting there is a classical parametrisation result which looks quite different to the one above. It includes two infinite families of parameters as well as the parametrisation map going in the opposite direction: from parameters to Toeplitz matrices. This parametrisation theorem was originally conjectured by Schoenberg \cite{Schoenberg:48} with proof later given in \cite{ASW, AESW, Edrei52}. We refer to the two infinite  parameter sequences in this parametrisation theorem as the `Schoenberg parameters' and to the theorem itself as Edrei's theorem, for his key contribution given in \cite{Edrei52}.  Let us call a sequence $(c_0,c_{1},c_2,\dots)$ of nonnegative real numbers \textit{totally nonnegative} if the infinite upper-triangular Toeplitz matrix $(c_{{i-j}})$ is totally nonnegative.
 
 \begin{theorem}[Edrei's theorem]\label{theorem:edrei-infinite}
 A sequence $(1, c_{1},c_2,\dots)$ is totally nonnegative if and only if its generating function is of the form
  \begin{equation}
  \label{eq:inf-gen-func}
  1 + c_1x + c_2x^2 + \cdots = e^{\gamma x}\prod\limits_{i=1}^{\infty}\frac{1 + \beta_{i}x}{1-\alpha_i x},
  \end{equation}  for some $((\boldsymbol{\alpha},\boldsymbol{\beta}),\gamma)\in\Omega_S\times\R_{\ge 0}$ where $\Omega_S$ denotes the parameter space
\[
\Omega_{S}=\left\{(\boldsymbol{\alpha},\boldsymbol{\beta})\in\R_{\ge 0}^{\N}\times \R_{\ge 0}^\N\left|\begin{array}{l} \boldsymbol{\alpha}={(\alpha_i)}_{i\in\N} \text{ with } \alpha_{i}\ge \alpha_{i+1} \text{ and }\sum_{i}\alpha_i<\infty\\ \boldsymbol{\beta}={(\beta_j)}_{j\in\N} \text{ with }  \beta_{j}\ge \beta_{j+1} \text{ and }\sum_{j}\beta_j<\infty\end{array}\right.\right\}. 
\]
\end{theorem}

Further normalised ($c_0=c_1=1$) infinite totally positive Toeplitz matrices and their Schoenberg parameters are classically related to irreducible normalised characters of the infinite symmetric group, by work of Thoma~\cite{Thoma} from the 1960's. Moreover, these characters are limits of normalised characters $\chi_{\lambda_n}/(\chi_{\lambda_n}(e))$ of $S_n$ as $n\to\infty$, and subsequent work of Vershik and Kerov~\cite{KEROVVERSHIK1981} from the 1980's interprets the $\alpha_i$ and $\beta_i$ as limits of  normalised arm and leg-lengths of the Young diagrams $\lambda_n$. It is interesting to observe the breadth of perspectives on Edrei's theorem, whose original proof belongs to the field of analysis and Nevalinna theory, but which has such clear representation theoretic significance. We note also that the theory of totally positive sequences, or Polya frequency sequences as they are sometimes referred to, is wide-ranging in itself. For example, it has a continuous analogue, and there are even reformulations of the Riemann hypothesis in terms of total positivity~\cite{Katkova,Grochenig2023}. We mention \cite{BorodinOlshanskiBook, Karlin} as further references concerning classical total positivity and Thoma's related theory. 

Our paper will build on the recently discovered relationship between Edrei's theorem and the parametrisation of finite totally positive Toeplitz matrices via quantum cohomology rings of full flag varieties $SL_{n}/B_{n}=\Fl_n(\C)$, that was given in \cite{rietsch2025totallypositivetoeplitzmatrices}, building on \cite{rietsch2025tropicaltoeplitzmatricesparametrisations}, and that goes as follows.

  \begin{theorem}\label{t:fullflag}\cite{rietsch2025totallypositivetoeplitzmatrices}
Consider a sequence $(u^{(n)})_n$ with $u^{(n)}\in X_{{B_n}}(\R_{>0})$. Suppose $u^{(n)}$ converges uniformly to an infinite totally positive Toeplitz matrix $u_\infty(\mathbf c)$ that has Schoenberg parameters 
$\alpha_1>\alpha_2>\dotsc >0$ and $\beta_1>\beta_2>\dotsc >0$. For $w\in S_\infty = \cup_{n\in \mathbb N}S_n$, consider $\mathfrak S_{B_n}^{w}\in \C[X_{B_{n}}]$ given by  $\mathfrak S_{B_n}^{w}:= \Phi_{B_n}(\sigma^w)$ using Theorem~\ref{t:Peterson}, whenever $n$ is large enough. Let $\hat{w}_{n}$ denote the  permutation that is conjugate to $w\in S_n$ by the longest element $w_0^{(n)}$ in $S_n$. Then we have,
\begin{eqnarray}
\lim_{n\to\infty}\mathfrak S_{B_n}^w(u^{(n)})&=&  S_w\left(\frac{1}{\alpha_1},\frac{1}{\alpha_2},\frac{1}{\alpha_3},\dotsc\right),\label{e:limSw}\\
\lim_{n\to\infty}\mathfrak S_{B_n}^{\hat{w}_{n}}(u^{(n)})&=&S_w\left(\frac{1}{\beta_1},\frac{1}{\beta_2},\frac{1}{\beta_3},\dotsc\right),\label{e:limSwn}
\end{eqnarray}
where $S_w$ is the classical Schubert polynomial 
associated to $w$. Moreover, for the $n$-th roots of the quantum parameters $q_i^{(n)}:=\Phi_{B_n}(q_i)$ we have
\begin{eqnarray}
\lim_{n\to\infty}\sqrt[n]{q^{(n)}_i(u^{(n)})}&=&\frac{\alpha_{i+1}}{\alpha_i}\label{e:limq1}\\
\lim_{n\to\infty}\sqrt[n]{q^{(n)}_{n-i}(u^{(n)})}&=&\frac{\beta_{i+1}}{\beta_i}
\label{e:limq2}
\end{eqnarray}
\end{theorem}
The main goal of the paper is to determine  analogous asymptotical formulas in the Grassmannian setting. The tools in this case are quite different from the ones used in \cite{rietsch2025totallypositivetoeplitzmatrices}, since for Grassmannians we are able to employ rather explicit descriptions of the Peterson Toeplitz variety. However,  \cite{rietsch2025totallypositivetoeplitzmatrices}  was critically involved in the process of choosing the correct conventions, that we expect should allow  generalisation of our results to partial flag varieties more generally. An additional result concerns the strange duality automorphism. This is an involution of (localised) Grassmannian quantum cohomology that was constructed independently by Postnikov and Hengelbrock. Our main results in this paper can be summarised as follows. 
\begin{enumerate}
    \item After careful setting up of conventions we provide analogues of Theorem~\ref{t:fullflag} for  $Gr(n-k,n)$ and  $Gr(k,n)$ and prove associated asymptotic formulas for Schubert classes in two different ways. 
    \item For the Peterson-Toeplitz variety  $X_{P_n}^{(2n)}$ of $Gr(n,2n)$, we prove that the asymptotic limiting set of $X_{P_n}^{(2n)}(\R_{>0})$, as $n\to\infty$, corresponds precisely to the set of exponential generating functions $\{e^{\gamma x}\mid \gamma\in R_{\ge 0}\}$ in terms of Edrei's theorem.
    \item We give a new interpretation of the   strange duality automorphism of $qH^*(Gr(k,n),\C)[q\inv]$ as part of Peterson theory. This interpretation is used in one of the proofs from (1) but may be of independent interest. 
    \item We give conjectural analogues of our results for more general partial flag varieties, and prove a result concerning the asymptotics of the quantum parameters in such a setting. 
\end{enumerate}

\subsection{Structure of the paper} This paper is structured as follows. Section~\ref{s:conventions} establishes background and conventions for the chapters that follow and constructs a useful equivalence between two different versions of Peterson's isomorphisms. Section~\ref{sec:preliminaries} discusses the theory of quantum cohomology for Grassmannians, and specifically its relationship with total positivity. In Section~\ref{sec:asymptquantparam}, we explore the asymptotic behaviour of the totally positive part of the Peterson Toeplitz variety of the Grassmannian for $Gr(n-m,n)$ as $n\to\infty$ (with $m$ fixed). Then in  Section~\ref{sec:asymptquantparam2} we consider the case of $Gr(k,n)$ with $k$ fixed, and in Section~\ref{sec:gr2n} the case of $Gr(n,2n)$. In Section~\ref{s:strange} we present a new perspective on the `strange duality' of Grassmannian quantum cohomology, and use it to give an alternative proof of the asymptotics of Schubert classes from Sections~\ref{sec:asymptquantparam} and \ref{sec:asymptquantparam2}. Finally, in Section~\ref{sec:conjectures}, we prove some partial results and formulate conjectures concerning the generalisation to partial flag varieties $SL_n/P$.

\section{Background and conventions}\label{s:conventions}
There are many different conventions concerning the Peterson presentation of quantum cohomology rings. In \cite{rietsch2025totallypositivetoeplitzmatrices} we chose conventions for the full flag variety $\Fl_n$ and its associated Peterson Toeplitz variety that lent themselves  particularly well to taking the infinite limit. We extend these conventions to partial flag varieties here, with the only change being that everything is transposed, so that our Toeplitz matrices are lower-triangular while the ones in \cite{rietsch2025totallypositivetoeplitzmatrices} were upper-triangular. Our choices thereby become closer to \cite{Kostant:qcoh,rietsch2001flagvarieties}, where it is also lower-triangular Toeplitz matrices that appear, in line with \cite{LRY:Schubert}. 

\subsection{Notations associated to $SL_n$ and the variety $X_P$} Let $G=SL_n$, and write $\mathfrak n_-$ for  the Lie algebra of $U_-$ and $\mathfrak u_+$ for the Lie algebra of $U_+$. Let  $F=\sum_{i\in I} F_i$ be the sum of the duals, $F_i=e_i^*$, of the positive Chevalley generators $e_i\in\mathfrak u_+$. We may call $F$ the standard principal nilpotent, noting that it lies in the dual,  $\mathfrak g^*$, of the Lie algebra $\mathfrak g$. Let us consider the one-parameter subgroups $x_i:\C\to U_+$ and $y_i:\C\to U_-$ given by 
\[
x_i(t)=\exp(t e_i), \quad y_i(t)=\exp(t f_i),\qquad t\in\C, i\in I,
\]
where the $f_i\in\mathfrak u_-$ are the negative Chevalley generators. Let $\Pi=\{\alpha_i\mid i\in I \}$ denote the set of simple roots (note that $\alpha_i$ is also used throughout this paper to denote the Schoenberg parameters of a Toeplitz matrix, but it should be clear from context which $\alpha_i$ is being used). As before, the Weyl group $W=S_n$ and is generated by the simple reflections. For any $w\in W$ we choose the representative $\bar w\in N_G(T)$ determined by 
\begin{equation}\label{e:wbar}
\bar s_i:=y_i(1)x_i(-1)y_i(1)\qquad \text{and}\qquad\bar w:=\bar s_{i_1}\dotsc \bar s_{i_n}, 
\end{equation}
for a/any reduced expression $w=s_{i_1}\dotsc s_{i_n}$ of $w$. We also consider another representative $\dot w:=\dot s_{i_1}\dot s_{i_2}\dotsc \dot s_{i_n}$ where $\dot s_i:=\bar s_i\inv$. For any parabolic subgroup $P\supset B$ we have the associated subgroup $W_P=\langle s_i\mid i\in I_P\rangle$ and the set  $W^P$ of  minimal coset representatives for $W/W_P$. The longest element in $W $ is denoted $w_0$ and the longest element in $W_P$ is $w_P$. 

Recall that $X=U_-^F$ is the stabilizer subgroup for $F$, and the  Peterson Toeplitz variety associated to a parabolic $P$ in our conventions is
\begin{equation}\label{e:PetersonToeplitzConv}
X_P=X\cap B_+w_0w_PB_+=\{u\in U_-\cap B_+ w_0w_PB_+\mid u\cdot F=F\}.
\end{equation}
If $P=B$ then $X_B=X\cap B_+ w_0 B_+$ consists of the unipotent lower-triangular Toeplitz matrices $u$ for which the lower left-hand corner $i\times i$-minor $\Delta_i(u)$ is nonzero for all $i\in I=\{1,\dotsc, n-1\}$.

\subsection{Grassmannian permutations and associated functions}\label{s:YoungDiag}
Recall that $w\in S_n$ is called a Grassmannian permutation of descent $k$ if $k\in[n]$ is unique with $w(k)>w(k+1)$. Equivalently, $w$ is Grassmannian of descent $k$ if it lies in  $W^{P_k}$ for the maximal parabolic subgroup $P_k$ with $I_{P_k}=I\setminus \{k\}$. It then also lies in $W^P$ for any parabolic subgroup $P\subseteq P_k$. 
\begin{definition}\label{d:wminornew}  For a Grassmannian permutation $w$, let $\Delta^w_k$ denote the $k\times k$ minor with column set $\{w(1),\dotsc,w(k)\}$ and row set $[k]+n-k=\{n-k+1,\dotsc,n-1, n\}$. 
\end{definition}
We may use the standard combinatorics of partitions $\lambda$, specifically $\lambda=(\lambda_1,\dotsc,\lambda_k)$ with $\lambda_1\le n-k$, to describe the Grassmannian permutations $w\in W^{P_k}$ and their associated minors more conveniently. Let us consider $\lambda$ as identified with its Young diagram (that we also denote $\lambda$). We write $\lambda'$ for the conjugate partition, where $\lambda'_i$ is the length of the $i$-th column of $\lambda$. Suppose that $\lambda$ is fitted inside a bounding $k\times (n-k)$ rectangle. Then the associated Grassmannian permutation, denoted $w_\lambda$ or $w_\lambda^{(k)}$, is the permutation of $[n]$ given by
\begin{equation}\label{e:wlambda}
w_\lambda=
\begin{pmatrix}
1 &2& \dotsc & k & k+1&k+2&\dotsc &n \\
\lambda_k+1&\lambda_{k-1}+2&\dotsc &\lambda_1+k& k-\lambda'_1+1&k-\lambda'_2+2&\dotsc& n-\lambda'_{n-k}
\end{pmatrix}. 
\end{equation}
Equivalently, we may consider the SE border of $\lambda$ inside its bounding rectangle as an $n$-step path from the SW corner to the NE corner consisting of a sequence of northward and eastward steps. Let $i_1<i_2<\dotsc < i_k$ index the northward steps, and $i_{k+1}<i_{k+2}<\dotsc < i_{n}$ the eastward steps of this path. Then the permutation $w_\lambda$ has the property that $w_\lambda(\ell)=i_\ell$ for $\ell\in [n]$. Note that $w_0 w^{(k)}_\lambda w_0\inv=w^{(m)}_{\lambda'}$, where $m=n-k$, and the superscript is keeping track of the descent.

We can now describe the $k\times k$-minor $\Delta_k^\lambda:=\Delta_k^{w_\lambda}$ explicitly in terms of the Young diagram $\lambda$ by
\begin{equation}\label{e:Deltalambda}\Delta_k^\lambda(u)=\Delta_k^{w_\lambda}(u)=\langle  u \bar w_\lambda\cdot v_{1}\wedge \cdots \wedge v_k,v_{n-k+1}\wedge \cdots \wedge v_{n} \rangle=\langle u \cdot  v_{i_1}\wedge v_{i_2}\wedge \cdots \wedge v_{i_k} ,v_{n-k+1}\wedge \cdots \wedge v_{n}\rangle,
\end{equation}
where $i_j=\lambda_{k-j+1}+j$. Here $\{v_i\}$ is the standard basis of $\C^n$ and the formula is expressing $\Delta_k^\lambda(u)$ as a matrix coefficient in the fundamental representation $V_{\omega_k}=\bigwedge^k\C^n$.  
\begin{equation}\label{e:Deltakw}
\Delta_k^{w_\lambda}(u)=\langle  u \bar w_\lambda\cdot v_{1}\wedge \cdots \wedge v_k,v_{n-k+1}\wedge \cdots \wedge v_{n} \rangle.
\end{equation}
\begin{definition}\label{d:Swlambda}For $w$ a Grassmannian permutation of descent $k$ define a rational function $\mathfrak{S}^{w_\lambda}\in \C(X)$ by
\[
\mathfrak{S}^{w_\lambda}(u)\coloneq \frac{\Delta_k^{w_\lambda}( u)}{\Delta_k(u)}.
\]
Note that $\mathfrak S^{w_\lambda}$ is regular on $X_P$ if and only if $k\in I^P$, or equivalently, if the Grassmannian permutation $w_\lambda$ lies in $W^P$. In particular each of these  $\mathfrak S^{w_\lambda}$ is regular on $X_B$ (which is where all the $\Delta_k$ are nonvanishing). We write 
\begin{equation}\label{e:SwP}
\mathfrak S^{w_\lambda}_P:=\mathfrak S^{w_\lambda}|_{X_P}\quad \text{if $ w_\lambda\in W^P$}
\end{equation} to obtain a regular function $\mathfrak S^w_P\in \C[X_P]$ for any such Grassmannian permutation. 
While $\mathfrak S^w_P$ will always be a function on $X_P$, we may use the notation $\Delta^w_k$ to refer to the $k\times k$ minor associated to $w\in W^{P_k}$  more generally, as a function on $SL_n$. 
\end{definition}
\begin{remark}\label{r:transpose}
We have $\Delta_k=\Delta_k^e$ for the identity element  $e\in W$, and therefore $\mathfrak S^e(u)=1$.  Moreover, if $w=s_k$ then $\mathfrak S^{s_k}(u)=\Delta^{s_k}_k(u)/{\Delta_k(u)}$ where the numerator $\Delta^{s_k}_k(u)$ is the minor of $u$ with column set $[k-1]\cup\{k+1\}$ and row set $[k]+n-k$. Note that in terms of the transpose $u_+:=u^T$ this numerator would be the minor with row set $[k-1]\cup\{k+1\}$ and column set  $[k]+n-k$, consistent with \cite[(2.5)]{rietsch2025totallypositivetoeplitzmatrices}.
\end{remark}

 \subsection{{The Peterson variety $Y_P$ and $qH^*(SL_n/P)$}} We recall Dale Peterson's construction of the quantum cohomology rings of the partial flag varieties $SL_n/P$, see  \cite{peterson,rietsch2001flagvarieties,Kostant:qcoh}. First, we define the full Peterson variety. It has two equivalent descriptions,
\begin{equation}\label{e:Y}
Y:=\{gB_-\mid (g\inv\cdot F)|_{[\mathfrak n_-,\mathfrak n_-]}=0\}=\overline{U^F_-w_0B_-/B_-.}%=\overline{\{uw_0B_-\mid u\in U^F_-\}}.
\end{equation}
Here again $F=\sum_{i=1}^{n-1}e_i^*$, and the action of $G$ in the definition is by the coadjoint representation. The variety $Y$ is an $(n-1)$-dimensional irreducible subvariety of the flag variety, and by the second description can be considered as a compactification of $X=U_-^F$.  
We also set $Y^\circ:=Y\cap B_-w_0B_-$, and to any parabolic subgroup $P\supseteq B $ we associate 
\begin{equation}\label{e:Petersonstrata}
Y_P:=
Y\cap B_+w_PB_-/B_-\qquad \text{and}\qquad Y_P^\circ:=Y_P\cap B_-w_0B_-/B_-.
\end{equation}
In general type, the definition of the (full) Peterson variety $Y$ is analogous, but with $Y$ lying in the Langlands dual flag variety. Note that therefore one would usually consider $PSL_n$ rather than $SL_n$ when making the constructions in type $A$, but since this has no effect on the flag variety or the Peterson Toeplitz variety, we will work with $SL_n$ throughout. 

\begin{definition}\label{d:Gwlambda}
    Suppose $w_\lambda$ is a Grassmannian permutation of descent $k$. We let $G^{w_\lambda}$  denote the rational function on $G/B_-$ defined by 
\begin{equation}\label{e:Gwlambda}
G^{w_\lambda}(gB_-):=\frac{\langle g\cdot v_{k+1}\wedge\dotsc\wedge v_{n},\dot w_\lambda\cdot v_{k+1}\wedge\dotsc \wedge v_{n}\rangle}{\langle g\cdot v_{k+1}\wedge\dotsc\wedge v_{n}, v_{k+1}\wedge\dotsc \wedge v_{n}\rangle},
\end{equation}
and $G^{w_\lambda}_Y$ its restriction to the Peterson variety $Y$. If $P$ is a parabolic subgroup with $k\in I^P$, so that $w_\lambda\in W^P$, then the restriction of $G^w$ to $Y_P$ is a regular function and we denote it by $G^{w_\lambda}_P$, or possibly $G^{\lambda}_P$ when the descent $k$ is fixed.
\end{definition}
We can now state Peterson's theorem from \cite{peterson}. 
\begin{theorem}[{\cite{peterson,rietsch2001flagvarieties,rietsch2001err}}]\label{t:PetersonY} 
The full Peterson variety is comprised of the strata from \eqref{e:Petersonstrata}, that is, $Y(\C)=\bigsqcup_P Y_P(\C)$.
Moreover, we have compatible isomorphisms,
\begin{equation}\label{e:PetIsoP}qH^*(SL_n/P)\cong \C[Y_P] \qquad \text{and}\qquad qH^*(G/P)[q_{n_1}\inv,\dotsc, q_{n_k}\inv]\cong \C[Y_P^\circ],
\end{equation}
determined by $\sigma^{w_\lambda}_P\mapsto G^{w}_P$ and $\sigma^{w_\lambda}_P\mapsto G^{w}_P|_{Y_P^\circ}$, respectively, 
where $w_\lambda\in W^{P_k}$ with $k\in I^P$. 
\end{theorem}
\subsection{A comparison result} We now explicitly describe the change of conventions that relate the version of Peterson's theorem from the introduction (Theorem~\ref{t:Peterson}) to the original version (Theorem~\ref{t:PetersonY}).
\begin{definition}\label{d:iota} Let $\iota$ be the positivity-preserving involutive anti-automorphism of  $G=SL_n$ defined by
\[
\iota(x_i(t))=x_i(t),\quad \iota(y_i(t))=y_i(t),\quad \iota(d)=d\inv,
\]
where $d\in T$ is a diagonal matrix, see \cite[(1.7)]{BerensteinZelevinsky}. 
\end{definition}
\begin{proposition}\label{p:comparison}
The map $\mu: U_- \to B_-w_0B_-/B_-$ defined by $
    u \mapsto \iota(u) w_0B_-$ has the property that
\begin{equation}\label{e:mustarformula}
    \mu^*(G^{w_\lambda})=\frac{\Delta^{w_\lambda}_k}{\Delta^e_k}
\end{equation}
for any Grassmannian permutation $w_\lambda=w_\lambda^{(k)}$ in $S_n$. Moreover, the restriction of $\mu$ to the Peterson Toeplitz variety $X_P$ defines an isomorphism $\mu_P: X_P\to Y_P^\circ$ with the property  that $\mu_P^*(G_P^{w_\lambda})=\mathfrak S_P^{w_\lambda}$ whenever $w_\lambda\in W^P$.
\end{proposition}
\begin{remark}\label{r:iota} This proposition implies that the isomorphism $\mu_P$ between the Peterson Toeplitz variety $X_P$ and the Peterson variety $Y^\circ_P$ induces a map of the coordinate rings that relates the  
respective isomorphisms with quantum cohomology coming from \cite{peterson,rietsch2001flagvarieties} and from  Theorem~\ref{t:Peterson},
\begin{equation}\label{e:conventionsdiagram}
\begin{tikzcd}
	{\C[Y_P^\circ]} && {\C[X_P]} \\
	{qH^*(SL_n/P)[q_{j}\inv;j\in I^P]} && {qH^*(SL_n/P)[q_{j}\inv;j\in I^P].}
	\arrow["\mu_P^*",from=1-1, to=1-3]
\arrow["\text{\cite{peterson,rietsch2001flagvarieties}}", from=2-1, to=1-1]
\arrow["{\rm{Theorem}~\ref{t:Peterson}}"', from=2-3, to=1-3]
	\arrow["id", from=2-1, to=2-3]
\end{tikzcd}\end{equation}
 Note that we could have alternatively defined an isomorphism $X_P\to Y_P^\circ$ by $u\mapsto u\inv w_0 B_-$ instead of using $\iota$. This would have had the effect of replacing the identity map in \eqref{e:conventionsdiagram} by the sign map $\epsilon(\sigma^{w}_P)=(-1)^{\ell(w)}\sigma^{w}_P$, see the proof of Proposition~\ref{p:comparison}.
\end{remark}

To check Proposition~\ref{p:comparison} we will need the following straightforward lemma. 
\begin{lemma}\label{l:lambda'}
Let $w_\lambda\in W^{P_k}$ where $\lambda\subseteq k\times m$ with $k+m=n$. Let $\lambda'\subseteq m\times k$ denote the conjugate partition. We have that
\begin{equation*}\Delta_m^{\lambda'}(u)=
\Delta^{\lambda}_k(\bar w_0\inv (u\inv)^T \bar w_0)
\end{equation*}
for any $u\in SL_n$.
\end{lemma}
\begin{proof}
The proof uses the isomorphism $\bigwedge\nolimits^m\C^n\cong \left(\bigwedge\nolimits^k\C^n\right)^*$  
of representations of $SL_n$, which is explicitly given by
\[
\bar w\cdot v_1\wedge\dotsc \wedge v_m \ \mapsto\ \bar w \cdot (v_{n-k+1}\wedge v_{n-k+2}\wedge\dotsc\wedge v_n)^*, 
\]
where $w\in W^{P_m}$. Note that the identity $w_0\inv w_{\lambda'}w_0= w_\lambda $ lifts to $ \bar w_0\inv\bar w_{\lambda'}\bar w_0=\bar w_{\lambda}$ since $\bar w_0\bar w_\lambda\inv=\overline{w_0 w_\lambda\inv}=\overline{w_{\lambda'}\inv w_0}=\bar w_{\lambda'}\inv\bar w_0$, and note also that $(\bar w\inv)^T=\bar w$. We now have
\begin{eqnarray*}
\Delta^{\lambda'}_m(u)&=&\langle\, u \bar w_{\lambda'} \cdot  (v_{n-k+1}\wedge v_{n-k+2}\wedge\dotsc\wedge v_n)^*,\bar w_0 \cdot  (v_{n-k+1}\wedge v_{n-k+2}\wedge\dotsc\wedge v_n)^*\,\rangle\\
&=&\langle\, \bar w_0^T u \bar w_{\lambda'} \cdot  (v_{n-k+1}\wedge v_{n-k+2}\wedge\dotsc\wedge v_n)^*,  (v_{n-k+1}\wedge v_{n-k+2}\wedge\dotsc\wedge v_n)^*\,\rangle\\
&=& (v_{n-k+1}\wedge v_{n-k+2}\wedge\dotsc\wedge v_n)^*( \bar w_{\lambda'}\inv u\inv (\bar w_0^T)\inv \cdot v_{n-k+1}\wedge\dotsc \wedge v_n)\\
&=&(\bar w_0\cdot v_{1}\wedge\dotsc\wedge v_k)^*( \bar w_{\lambda'}\inv u\inv (\bar w_0^T)\inv \cdot v_{n-k+1}\wedge\dotsc \wedge v_n)\\
&=&\langle
\bar w_0\cdot v_{1}\wedge\dotsc\wedge v_k,\bar w_{\lambda'}\inv u\inv (\bar w_0^T)\inv \cdot v_{n-k+1}\wedge\dotsc \wedge v_n\rangle\\
&=&\langle
(\bar w_0)\inv(u\inv)^T(\bar w_{\lambda'}\inv)^T\bar w_0\cdot v_{1}\wedge\dotsc\wedge v_k, v_{n-k+1}\wedge\dotsc \wedge v_n\rangle\\
&=&\langle
(\bar w_0)\inv(u\inv)^T\bar w_0\bar w_0\inv \bar w_{\lambda'}\bar w_0\cdot v_{1}\wedge\dotsc\wedge v_k, v_{n-k+1}\wedge\dotsc \wedge v_n\rangle\\
&=&\langle(\bar w_0)\inv(u\inv)^T\bar w_0 \bar w_{\lambda}\cdot v_{1}\wedge\dotsc\wedge v_k, v_{n-k+1}\wedge\dotsc \wedge v_n\rangle = \Delta^\lambda_k(\bar w_0\inv(u\inv)^T\bar w_0),
\end{eqnarray*}
which proves the lemma.
\end{proof}
%\K{ is actually $\iota(u)=\bar w_0\inv (u\inv)^T \bar w_0$ when $u $ is Toeplitz and could this simplify the proof below? .}
\begin{proof}[Proof of Proposition~\ref{p:comparison}] Let us first check that $\mu$ sends $X_P$ to $Y_P^\circ$. It suffices to show that $\iota$ sends $X_P=X\cap B_-w_0w_P B_-$ to $X_Q=X\cap B_-w_P w_0 B_-$, since $Y^\circ=Xw_0B_-$ and $Y_P^\circ=(X\cap B_+w_Pw_0 B_+)w_0B_-$.   
Passing to $GL_n$, let $d_\epsilon$ be a diagonal matrix with alternating entries $\pm 1$, so that $\alpha_i(d_\epsilon)=-1$ for all $i$. Then we note that $\iota(u)=d_\epsilon u\inv d_\epsilon\inv$. The coadjoint action also extends to $GL_n$ %to include $d_\epsilon$ %on $\mathfrak g^*$ 
and we have $d_\epsilon\cdot F=-F$, from which it follows that $\iota$ preserves $X=U_-^F$. The statement that $\iota(X_P)=X_Q$ follows from the fact that $\iota$ is an anti-automorphism. 
%Now $Y^\circ=U_-^Fw_0B_-/B_-$, and we therefore have that $\mu$ is a well-defined isomorphism. 

We now prove \eqref{e:mustarformula}. Consider $u\in U_-$ and use Lemma~\ref{l:lambda'} to re-express the quotient of minors 
\begin{eqnarray*}
    \nonumber
    \frac{\Delta_k^{\lambda}(u)}{\Delta_k(u)}
    &=&\frac{\Delta_{n-k}^{\lambda'}(\bar w_0\inv (u\inv)^T\bar w_0)}{\Delta_{n-k}(\bar w_0\inv (u\inv)^T\bar w_0)}
    =\frac{\langle \bar w_0\inv (u\inv)^T\bar w_0\bar w_{\lambda'}\cdot v_1\wedge\dotsc\wedge v_{n-k},\bar w_0\cdot v_1\wedge \dotsc\wedge v_{n-k}\rangle}{\langle \bar w_0\inv (u\inv)^T\bar w_0\cdot v_1\wedge\dotsc\wedge v_{n-k},\bar w_0\cdot v_1\wedge \dotsc\wedge v_{n-k}\rangle}
    \\ 
    \nonumber
   & =&\frac{\langle \bar w_0\bar w_{\lambda'}\cdot v_1\wedge\dotsc\wedge v_{n-k},u\inv (\bar w_0\inv)^T \bar w_0\cdot v_1\wedge \dotsc\wedge v_{n-k}\rangle}{\langle \bar w_0\cdot v_1\wedge\dotsc\wedge v_{n-k},u\inv (\bar w_0\inv)^T\bar w_0\cdot v_1\wedge \dotsc\wedge v_{n-k}\rangle}
    =\frac{\langle \bar w_{\lambda}\bar w_0 \cdot v_1\wedge\dotsc\wedge v_{n-k},u\inv \cdot v_1\wedge \dotsc\wedge v_{n-k}\rangle}{\langle \bar w_0\cdot v_1\wedge\dotsc\wedge v_{n-k},u\inv\cdot v_1\wedge \dotsc\wedge v_{n-k}\rangle}
    \\
    \nonumber
    &=&\frac{\langle \bar w_{\lambda}\cdot v_{k+1}\wedge\dotsc\wedge v_{n},u\inv \cdot v_1\wedge \dotsc\wedge v_{n-k}\rangle}{\langle v_{k+1}\wedge\dotsc\wedge v_{n},u\inv\cdot v_1\wedge \dotsc\wedge v_{n-k}\rangle}=
    \frac{\langle u\inv\dot w_0 \cdot v_{k+1}\wedge\dotsc\wedge v_{n}, \bar w_{\lambda}\cdot v_{k+1}\wedge\dotsc\wedge v_{n}\rangle}{\langle u\inv\dot w_0 \cdot v_{k+1}\wedge\dotsc\wedge v_{n},  v_{k+1}\wedge\dotsc\wedge v_{n}\rangle}
 \text{\qquad(\theequation)}\label{e:SwvsGw} \refstepcounter{equation}  \\
    &=& (-1)^{|\lambda|}G^{w_{\lambda}}(u\inv \dot w_0B_-).\nonumber
\end{eqnarray*}
Note that we have 
\[\bar w_{\lambda}\cdot v_{k+1}\wedge\dotsc\wedge v_{n}=(-1)^{|\lambda|}\,\dot w_\lambda\cdot v_{k+1}\wedge\dotsc\wedge v_{n},\]
which gives rise to the sign $(-1)^{|\lambda|}$ in the final equality. We now claim that this sign disappears if we replace $u\inv$ by $\iota(u)$. Namely, consider the Zariski-open subset of $U_-$ consisting of all elements of the form $y(\mathbf m)=y_{i_1}(m_1)\dotsc y_{i_N}(m_N)$, where $w_0=s_{i_1}\dotsc s_{i_N}$ is a fixed reduced expression and $\mathbf m=(m_1,\dotsc, m_N)\in (\C^*)^N$. We call this a `Lusztig torus' in $U_-$. Set
\begin{equation}
    D_\lambda(\mathbf m):= \langle \iota(y(\mathbf m))\dot w_0 \cdot v_{k+1}\wedge\dotsc\wedge v_{n}, \bar w_{\lambda}\cdot v_{k+1}\wedge\dotsc\wedge v_{n}\rangle,
\end{equation}
relating to the final numerator in \eqref{e:SwvsGw}. 
Then $D_\lambda$ is a homogeneous polynomial of degree $k (n-k)-|\lambda|$ in the $m_i$ coordinates (by weight-space considerations). Since $y((m_j)_j)\inv=\iota(y((-m_j)_j)$ we have 
\begin{equation}
    \langle y(\mathbf m)\inv\dot w_0 \cdot v_{k+1}\wedge\dotsc\wedge v_{n}, \bar w_{\lambda}\cdot v_{k+1}\wedge\dotsc\wedge v_{n}\rangle=D_\lambda(-\mathbf m)= (-1)^{k(n-k)-|\lambda|}D_\lambda(\mathbf m).
\end{equation}
It follows that 
\begin{equation}
\langle u\inv\dot w_0 \cdot v_{k+1}\wedge\dotsc\wedge v_{n}, \bar w_{\lambda}\cdot v_{k+1}\wedge\dotsc\wedge v_{n}\rangle=(-1)^{k(n-k)-|\lambda|} \langle \iota(u)\dot w_0 \cdot v_{k+1}\wedge\dotsc\wedge v_{n}, \bar w_{\lambda}\cdot v_{k+1}\wedge\dotsc\wedge v_{n}\rangle,
\end{equation}
for all $u\in U_-$. Applying this identity to both numerator and denominator (where $\lambda=\emptyset$) of $G^{w_\lambda}(u\inv \dot w_0 B_-)$, we see that 
\[
G^{w_\lambda}(u\inv \dot w_0 B_-)=(-1)^{|\lambda|}G^{w_\lambda}(\iota(u) \dot w_0 B_-).
\]
It follows that $\mu^*(G^{w_\lambda}|Y^\circ)=\mathfrak S^{w_\lambda}$, and the statements about the restriction $\mu_P$ to $X_P$ follow immediately.  
\end{proof}

\begin{remark}\label{r:iotaequiv} We note that the two positivity-preserving involutions $u\mapsto\iota(u)$ and $u\mapsto \bar w_0\inv (u\inv)^T \bar w_0$ of $U_-$ used above are very different in nature, with the first being an anti-automorphism while the second is an automorphism. However, for $u\in X$ we have the equality $\iota(u)=\bar w_0\inv (u\inv)^T\bar w_0$, as both sides represent the involution $p(x)\leftrightarrow \frac{1}{p(-x)} \mod x^{n+1}$ on the level of generating functions. Note that $X$ is an abelian subgroup of $U_-$, so automorphisms and anti-automorphisms  coincide.
 \end{remark}

\subsection{{Quantum parameters for $SL_n/P$}}\label{s:qparam} Consider the partial flag variety $SL_n/P=\Fl_{\mathbf n}(\C^n)$, where $\mathbf n=(n_1,\dotsc, n_r)$ is given by $I^P= \{n_i\mid i=1,\dotsc, r\}$. We may also write $P_{\mathbf n}$ for this parabolic $P$. Let us identify the coordinate ring $\C[X_P]$ with  $qH^*(G/P,\C)[q_{n_1}\inv,\dotsc,q_{n_r}\inv]$ via the isomorphism identifying the functions $\mathfrak S^w_P$  associated to Grassmannian permutations with descent $k\in I^P$ with their corresponding Schubert classes $\sigma^w_P$, see  Definition~\ref{d:Swlambda} and \cref{t:Peterson}.  
Under this isomorphism the quantum parameters $q_{n_i}$ are then described explicitly as follows.

Consider a factorisation of  $u\in X_P$ of the form $u=x_L\bar w_0\bar w_P\inv d x_R$ with $x_L,x_R\in U_+$ and $d\in T=T_{SL_n}$. Note that $d$ is unique (while $x_L,x_R$ are not in general).  From $u\in X_P$ it follows straightforwardly that $d\in T^{W_P}$. In the full flag variety case, $d$ is represented by the diagonal matrix appearing in Remark~\ref{r:diforfullflag}. In our conventions we have that, as functions on $X_P$, the quantum parameters are given by 
\begin{equation}\label{e:qni}
q_{n_i}(u)=\alpha_{n_i}(d)\inv.
\end{equation}
Moreover, if $u\in X_P(\R_{>0})$ then $d$ is in the positive part of $T^{W_P}$. We also note that (by a translation of \cite[Lemma~5.1]{rietsch2001flagvarieties}) the quantum parameters $q_{n_1},\dotsc, q_{n_r}$ from the quantum cohomology of $\Fl_{\mathbf n}(\C^n)$ 
 satisfy
\begin{equation}\label{e:qviadelta}
q_{n_i}^{(n_{i+1}-n_i)(n_i-n_{i-1})}=\frac{(\Delta_{n_{i-1}})^{n_{i+1}-n_i}(\Delta_{n_{i+1}})^{n_{i}-n_{i-1}}}{(\Delta_{n_{i}})^{n_{i+1}-n_{i-1}}},
\end{equation}
 viewed as functions on $X_{P}$. In the full flag variety case this gives $q_i=\frac{\Delta_{i-1}\Delta_{i+1}}{\Delta_i^2}$, which translates back to the formula due to Kostant~\cite{Kostant:qcoh}.

\begin{remark}\label{r:qcomparison}
Recall from \cite{rietsch2001flagvarieties,rietsch:mirror} that for $yw_0B_-\in Y^\circ_P$, factorised as $y=u_1 t \dot w_P\dot w_0\inv u_2$, the quantum parameter $q_{n_i}$ is given by $q_{n_i}(yw_0 B_-)=\alpha_{n_i}(t)$. Given $u\in X_P$ and supposing we have chosen a factorisation $u=x_L\bar w_0\bar w_P\inv d x_R$, as in Section~\ref{s:qparam}, then for the image  $\iota(u)w_0B_-\in Y_P^\circ$ we have that $y:=\iota(u)$ has the factorisation
\[
y=\iota(u)=\iota(x_R)d\inv\iota(\bar w_0\bar w_P\inv)\iota(x_L)=\iota(x_R)d\inv\bar w_P\inv\bar w_0\iota(x_L),
\]
which is of the form $u_1 t \dot w_P\dot w_0\inv u_2$ for $\iota(x_R)=u_1, \iota(x_L)=u_2$ and $t=d\inv$. (Note that $\iota(\bar s_i)=\bar s_i$ and $\dot w_P=\bar w_P\inv$ and $\dot w_0\inv=\bar w_0$.) The formula $q_{n_i}(u)=\alpha_{n_i}(d)\inv$ is therefore equivalent to the formula $q_{n_i}(y w_0 B_-)=\alpha_{n_i}(t)$ from \cite{rietsch2001flagvarieties,rietsch:mirror}. 
\end{remark}

\subsection{General Schubert classes}\label{s:Sw} A remarkable aspect of Peterson's theory is that for \textit{any} $w\in W$ there is one rational function ``$\mathfrak S^w$'' that is responsible for all  of the Schubert classes $\sigma_P^w$ arising in all the quantum cohomology rings $qH^*(G/P)[q_{n_1}\inv,\dotsc, q_{n_r}\inv]$.  If $w$ is a Grassmannian permutation then $\mathfrak S^w$ is the function given explicitly in Definition~\ref{d:Swlambda}. In general, it is determined by the (full flag) map
\begin{equation}\label{e:PetIsoB}
qH^*(SL_n/B,\C)[q_1\inv,\dotsc, q_{n-1}\inv]\overset\sim\longrightarrow \C[X_B]\hookrightarrow \C(X)
\end{equation}
coming from Theorem~\ref{t:Peterson}. 
It follows from \cite[Proposition~11.1]{rietsch2001flagvarieties} that if 
%there is a Schubert basis element in the associated quantum cohomology ring indexed by $w$, that is, if 
$w\in W^P$, then the rational function $\mathfrak S^w$ is regular on $X_P$, and its restriction to $X_P$ represents precisely the Schubert basis element $\sigma^w_P$. 
\begin{definition}\label{d:SwPgen}
For any $w\in W$ consider the rational function $\mathfrak S^w$ on $X$ obtained from $\sigma^w_B$ via \eqref{e:PetIsoB}. For $w\in W^P$ we obtain an element of $\C[X_P]$ by $\mathfrak S_P^w:=\mathfrak S^w|_{X_P}$. This regular function agrees with the image $\Phi_P(\sigma^w_P)$ of $\sigma^w_P$ under the Peterson Toeplitz isomorphism from Theorem~\ref{t:Peterson}.    
Compare Definition~\ref{d:Swlambda}.
\end{definition}
The following definition will be useful for extending the domains of the functions $\mathfrak S^w_B$. 
\begin{definition}\label{d:XJ}
For any subset $J\subseteq I$, let $X_{(J)}:=\{x\in X\mid \Delta_j(x)\ne 0\text{ for all $j\in J$}\}$. 
\end{definition}
\begin{remark}\label{r:XJ}
The variety $X_{(J)}$ collects together all of the Peterson Toeplitz varieties $X_{P_{\mathbf n}}$ corresponding to partial flag varieties $\Fl_{\mathbf n}$ for which the set of dimensions $\{n_i\}$ includes all  of $J$ 
\begin{equation}\label{e:XJ}
X_{(J)}(\C)=\bigsqcup_{\{\mathbf n\}\supseteq J} X_{P_\mathbf n}(\C).
\end{equation}
\end{remark}
Note that for every $w\in W$ there exists a unique parabolic subgroup $P$ such that $w\in W^P$ and which is maximal for this condition. Namely, if we let $J(w):=\{i\in I\mid w(\alpha_i)<0\}$, then this parabolic subgroup $P$ is determined by $I^P=J(w)$. 

\section{Grassmannian Toeplitz Matrices, quantum cohomology and total positivity}\label{sec:preliminaries}

We will now describe the variety $X_P$ and the isomorphism from \cref{t:Peterson} in more detail in the Grassmannian case, where $P=P_k$. We will make practical use of some results from \cite{rietsch2001grassmannians}. However, it is important to note that the isomorphism of $X_{P_k}$ with quantum cohomology described in that paper is a completely different one from the isomorphism in \cref{t:Peterson}, and in particular does not have analogues for other non-maximal parabolics. In this paper, the isomorphism  that we focus on is the one given in \cref{t:Peterson}.

\subsection{Grassmannian Toeplitz Matrices}\label{ssec:toeplitz-matrices}

We now consider $\mathrm{Gr}(k,n)=SL_n/P_k$, where we may fix $m\coloneq n-k$. % Note that $\mathrm{Gr}(k,n)\cong \mathrm{Gr}(m,n)$.
We have the associated Peterson Toeplitz variety 
\begin{equation}\label{e:XPk}
  X_{P_k}(\C) = \left\{ u=u_n(c_1,\dotsc, c_m)=\left .
  \begin{pmatrix}
    1& &&&& \\
    c_1&1 &&&&\\
    \vdots &c_1&&&&\\
    c_m & & \ddots &&&\\
     &\ddots &&c_1&1&\\
    0 & & c_m & \cdots &c_1&1\end{pmatrix} \,\right|\, \Delta_k(u) \neq 0 , \quad \Delta_{i\neq k}(u) = 0 \right\}.
\end{equation}
For every partition $\lambda\subseteq k\times m$ consider the $k\times k$ minor $\Delta^\lambda_k$ from Definition~\ref{d:wminornew} and \eqref{e:Deltalambda}. We replace $w_\lambda$ by $\lambda$ also in our notation for the functions $\mathfrak S_P^w$ from Definition~\ref{d:Swlambda}, 
\begin{equation}\label{e:Slambda}
\mathfrak S_{P_k}^\lambda(u):=\mathfrak S^{w_\lambda}_{P_k}(u)=\frac{\Delta_k^\lambda(u)}{\Delta_k(u)}.
\end{equation}
We may drop the $k$ from the notation if we are in the Grassmannian case $Gr(k,n)$ and the parabolic $P=P_{k}$ is fixed. Note that when $\lambda=\emptyset$ we have $w_\emptyset=e$ and $\mathfrak S_{P}^\emptyset=1$. On the other hand for $\lambda$ corresponding to the maximal Young diagram, the $k\times m$ rectangle, we have $\Delta^{k\times m}=1$ and  therefore $\mathfrak S_{P}^{k\times m}=1/\Delta_k$.
\subsection{Quantum Cohomology of $Gr(k,n)$}\label{s:qcoh} We start by recalling in more detail the original description of the quantum cohomology ring $qH^*(Gr(k,n))$, which goes back to  \cite{BERTRAM1997289,BDW,SiebertTian,WittenGrassmannian,BuchGrassmannianQCoh}. 
Let $t_1,\dotsc, t_k$ be the Chern roots of the dual tautological subbundle $\mathcal S_{k,n}^\vee$ of $Gr(k,n)$. Let $X_j=E_j(t_1,\dotsc, t_k)$ and $Y_r=H_r(t_1,\dotsc, t_k)$ denote the elementary and complete symmetric polynomials in the $t_i$'s, respectively. Thus $Y_r=\det(X_{i-j+1})_{i=1}^r$. The cohomology of the Grassmannian is the quotient of the ring of symmetric polynomials  $\Z[X_1,\dotsc, X_k]=\Z[t_1,\dotsc, t_k]^{S_k}$ by the ideal $I=(Y_{n-k+1},\dotsc, Y_n)$, and the quantum cohomology is the $1$-parameter deformation (since the Grassmannian has Picard rank $1$) given explicitly by 
\begin{equation}\label{e:qcoh}
qH^*(Gr(k,n),\Z)=\Z[X_1,\dotsc, X_k;q]/(Y_{n-k+1},\dotsc, Y_n+(-1)^k q).
\end{equation}
The Schubert classes are indexed by partitions $\lambda\subseteq k\times {(n-k)}$ and form a basis $\{\sigma^\lambda\}$ of  $qH^*(Gr(k,n),\Z)$ over $\Z[q]$. We may write $\sigma^\lambda=\sigma_{P_k}^\lambda$ to emphasise the Grassmannian $Gr(k,n)$ in question, but will generally leave out the subscript when the context is clear. 

It was proved by Bertram \cite{BERTRAM1997289} that the Schubert basis elements  
are represented by Schur polynomials.  
Namely, if $\lambda=(\lambda_1)$ then $\sigma^{(\lambda_1)}=Y_{\lambda_1}$, and for general $\lambda=(\lambda_1,\dotsc, \lambda_r)$, the quantum Schubert class $\sigma^\lambda$ satisfies the `quantum Giambelli formula'
 \begin{equation}
  \label{eq:q-giambelli}
   \sigma^{\lambda} = \det 
    \begin{pmatrix}
      Y_{\lambda_{1}} &  Y_{\lambda_{1} + 1} & Y_{\lambda_{1}+2} &\cdots & Y_{\lambda_{1} + r-1} \\
                          Y_{\lambda_{2}-1}  & Y_{\lambda_{2}} & & & \\
                         Y_{\lambda_{3}-2}   &   Y_{\lambda_{3}-1}& Y_{\lambda_{3}}& & \\
                          \vdots &  & &\ddots &\vdots \\
                       Y_{\lambda_{kr}- r+1}     &  &\cdots & & Y_{\lambda_{r}}
\end{pmatrix}
\end{equation}
inside $qH^*(Gr(k,n))$, where $Y_0:=1$ and $Y_{<0}:=0$. In other words, the usual Giambelli formula in $H^*(Gr(k,n))$ continues to hold without any $q$-deformation in the quantum cohomology ring. Since, inside the ring of symmetric polynomials, this is the Jacobi-Trudi formula for the Schur polynomial $S_\lambda$, we have that $S_\lambda(t_1,\dotsc, t_k)$ represents $\sigma^\lambda$. We note that this circumstance is particular to Grassmannians, while for more general partial flag varieties,  polynomials representing Schubert classes (quantum Schubert polynomials) need to be deformed to conform with quantum Schubert calculus \cite{FGP}.

The isomorphism $Gr(k,n)\cong Gr(m,n)$, where $m=n-k$, determines an isomorphism of quantum cohomology rings that sends $q_k$ to $q_m$ and $\sigma^\lambda_{P_k}$ to $\sigma^{\lambda'}_{P_m}$, see \cite{BCFF}. This also leads to dual versions of the quantum Giambelli and other formulas.
We recall a less standard presentation of the quantum cohomology ring that is more symmetric, which was  also given in \cite{BCFF},  
\begin{equation}\label{e:lambdakm}
\Lambda_{k,m}:=\Z[X_1,\dotsc, X_k;Y_1,\dotsc, Y_m;q]/(X_1-Y_1,X_2-X_1Y_1+Y_2,\dotsc, X_{k-1}Y_m-X_kY_{m-1},X_kY_m-q).
\end{equation} 
Namely, here $X_i$ and $Y_i$ represent the same classes as above, determining the isomorphism $qH^*(Gr(k,n),\Z)\cong\Lambda_{k,m}$.
We finally introduce a symmetric description of $qH^*(Gr(k,n),\C)[q\inv]$ as coordinate ring, that will be of use later on.
 \begin{definition}\label{d:Ykm} Let 
\begin{eqnarray}\label{e:Ykm}
\mathcal Y_{k,m}&:=&\{(z_1,\dotsc,z_k,z_{k+1},\dotsc, z_n)\in \C^n\mid z_1^n=\dotsc= z_n^n,\, z_i\ne z_j,\, z_i\ne 0 \}\big/S_k\times S_m,
\end{eqnarray}
where the $S_k$-factor acts by permuting the first $k$ coordinates, and $S_m$ permutes the remaining $m$. Note that the action of $S_k\times S_m$ is free, and $\mathcal Y_{k,m}$ is a (reducible) nonsingular affine curve. We denote an element of $\mathcal Y_{k,m}$ by $[z_1,\dotsc,z_n]_{k,m}$ to keep in mind that this expression is symmetric in the first $k$ and the last $m$ entries. We have the following lemma that will be very useful later on.
 \end{definition}
 \begin{lemma}\label{l:Ykm} We have an isomorphism 
 \begin{eqnarray*}
qH^*(Gr(k,n),\C)[q\inv]&\longrightarrow &\quad\qquad\C[\mathcal Y_{k,m}]\\
\sigma_{P_k}^\lambda \quad\qquad &\mapsto &S_\lambda(z_1,\dotsc, z_k)=S_{\lambda'}(-z_{k+1},\dotsc, -z_n).
 \end{eqnarray*}
In particular, under this isomorphism, $X_i$ maps to $E_i(z_1,\dotsc, z_k)$ and $Y_j$ to $E_j(-z_{k+1},\dotsc ,-z_n)$. Moreover $q$ maps to $E_n(z_1,\dotsc,z_k,-z_{k+1},\dotsc, -z_n)=(-1)^m z_1\dotsc z_n=(-1)^{k+1}z_1^n$.
\end{lemma}
\begin{proof}
Note that $z_1^n=\dotsc=z_n^n$ for distinct $z_i\ne 0$ is equivalent to the $z_i$ being a full set of roots of a polynomial of the form $x^n-c$ with $c\ne 0$. This condition can also be restated as saying that $\prod_{i=1}^n(1+z_i x)=1+(\prod z_i)x^n$. 

It is straightforward that the coordinate ring of $\mathcal Y_{k,m}$ is generated by the elementary symmetric functions $\mathbf X_i:=E_i(z_1,\dotsc, z_k)$ and $\mathbf Y_j:=E_j(-z_{k+1},\dotsc, -z_n)$. Moreover,  $\prod_{i=1}^n(1+z_i x)=1+(\prod z_i)x^n$ can be rewritten as
\begin{equation}\label{e:coeffcompare}
(1+\mathbf X_1 x+\mathbf X_2 x^2+\dotsc+\mathbf X_k x^k)(1-\mathbf Y_1 x+\mathbf Y_2 x^2+\dotsc +(-1)^m \mathbf Y_m x^m)=1+(-1)^m\mathbf X_k\mathbf Y_m x^n.
\end{equation}
The relations between the generators of $\C[\mathcal Y_{k,m}]$ are now given by comparing coefficients of $x^i$ on either side of \eqref{e:coeffcompare}. If we also set $\mathbf q:=\mathbf X_k\mathbf Y_m$, then it follows that $\C[\mathcal Y_{k,m}]$ has the presentation
\begin{equation}\label{e:CYkm}
\C[\mathbf X_1,\dotsc, \mathbf X_k;\mathbf Y_1,\dotsc,\mathbf Y_m;\mathbf q^{\pm 1}]/(\mathbf X_1-\mathbf Y_1,\mathbf X_2-\mathbf X_1\mathbf Y_1+\mathbf X_2,\dotsc,\mathbf X_{k-1}\mathbf Y_m-\mathbf X_k\mathbf Y_{m-1},\mathbf X_k\mathbf Y_m-\mathbf q).
\end{equation} 
This shows that we have an isomorphism $qH^*(Gr(k,n))\to\C[\mathcal Y_{k,n}]$ given by $X_j\mapsto \mathbf X_j$ and $Y_r\mapsto \mathbf Y_r$ and $q\mapsto \mathbf q$. But then we deduce that $\mathbf Y_j=E_j(-z_{k+1},\dotsc,-z_n)$ is also equal to $H_j(z_1,\dotsc, z_k)$. This extends to the identity $S_{\lambda'}(-z_{k+1},\dotsc,-z_n)=S_\lambda(z_1,\dotsc, z_k)$ by the dual and the regular Jacobi-Trudi formulas. Finally, $S_{\lambda'}(-z_{k+1},\dotsc,-z_n)=S_\lambda(z_1,\dotsc, z_k)$ is the image of  $\sigma^\lambda_{P_k}$ as claimed, by the  quantum Giambelli formula. It is immediate that $\mathbf q=\mathbf X_k\mathbf Y_m$ is given in terms of the $z_i$ by $(-1)^mz_1\dotsc z_n$. Setting $x=-z_1\inv$ in the identity  $\prod_{i=1}^n(1+z_i x)=1+(\prod z_i)x^n$ we deduce that $z_1\dotsc z_n=(-1)^{n+1}z_1^n$, which implies the second expression for $\mathbf q$ that was to be proved. 
\end{proof}
\begin{remark} In the context of Lemma~\ref{l:Ykm} we can now interpret the coordinates $z_1,\dotsc, z_k$ as the Chern roots of the dual tautological bundle $\mathcal S_{k,n}^\vee$ of $Gr(k,n)$, and the coordinates $z_{k+1},\dotsc, z_n$ as the Chern roots of the dual tautological quotient bundle $\mathcal Q_{k,n}^\vee$. The involution
\begin{eqnarray}\label{e:YkmSymm}
\mathcal Y_{k,m}&\overset\sim\longrightarrow & \mathcal Y_{m,k}
\\ \nonumber
{[} y_1,\dotsc , y_n ]_{k,m}&\mapsto &[-y_n,\dotsc,-y_1]_{m,k}.
\end{eqnarray}
induces the isomorphism between $qH^*(Gr(k,n))$ and $qH^*(Gr(m,n))$ from \cite{BCFF} that swaps $\sigma_{P_{k}}^{\lambda}$ and $\sigma_{P_m}^{\lambda'}$ and maps one quantum parameter to the other one.    
\end{remark}
 
\subsection{Total positivity and  ${X}_{P_k}$}\label{sec:totally-posit-part}
For now, $SL_n$ remains fixed and $SL_n/P_k=Gr(k,n)$. We recall that Peterson's isomorphism from~\cref{t:Peterson} sends the Schubert class $\sigma^\lambda$ to $\mathfrak{S}^{\lambda}_{P_k}$. Note that since $\mathfrak S^{\lambda}_{P_k}$ is a quotient of minors it takes values in $\R_{>0}$ on $X_{P_k}(\R_{>0})$. We now recall the explicit description of the totally positive part of $X_{P_k}$ from~\cite{rietsch2001grassmannians}, with some of the notation adapted to our conventions. First, 
given a sequence $\mathbf c=\left( c_1,\dots,c_m \right)$ of coefficients, let us again write $u_n(\mathbf c)$ for the $n\times n$ lower-triangular unipotent Toeplitz matrix with below-diagonal entries determined by $\mathbf c=(c_1,\dotsc,c_m)$, as in  \eqref{e:XPk}. Set $p_u(x)= 1 + c_1x + \cdots + c_{m} x^m$  to be the  generating polynomial associated to $u=u_n(c_1,\dotsc, c_m)$. 

\begin{definition}\label{d:unpu}
We define 
\begin{equation*}
u_{n}^{E}(x_{1},\dots, x_m):=u_n(E_1(\mathbf x),\dotsc, E_m(\mathbf x))=\begin{pmatrix}
 1 &  & &  &        & \\
  E_{1}(\mathbf x) &1   &    &       & & \\
  \vdots &E_{1}(\mathbf x) &  \ddots &  &      &  \\
   E_{m}(\mathbf x)&     &      \ddots  & \ddots &   &  \\
   &   \ddots  &       &    E_{1}(\mathbf x)  &  1    &  \\
   0&     &     E_{m}(\mathbf x)  &     \cdots  &   E_{1}(\mathbf x)    &1
\end{pmatrix},
\end{equation*}
which, equivalently, is the Toeplitz matrix $u=u_n(c_1,\dotsc, c_m)$ with  $p_u(x)=(1+x_1x)(1+x_2x)\dotsc(1+x_m x)$.
\end{definition}

\begin{definition}\label{d:Xnm} Define
\begin{eqnarray}\mathcal X_{n,m}&:=&\{(x_1,\dotsc, x_m)\in \C^m\mid x_1^n=\dotsc= x_m^n,\, x_i\ne x_j,\, x_i\ne 0 \}\big/S_m.
\end{eqnarray}
We write $\{x_1,\dotsc, x_m\}$ for the element of $\mathcal X_{n,m}$ represented by $(x_1,\dotsc, x_m)$. Note that $\mathcal X_{n,m}$ is again a reducible affine algebraic curve.   
\end{definition}

\begin{lemma}[{\cite[Lemma 3.7]{rietsch2001grassmannians}}]\label{l:Toeplitz-polynomial}\label{l:uE}
We have an isomorphism 
\[
\begin{array}{cccc}
u^E:&\mathcal X_{n,m} & \longrightarrow & X_{P_{n-m}}\\
&\{x_1,\dotsc,x_m\}&\mapsto & u_n^E(x_1,\dotsc, x_m).
\end{array}
\]
\end{lemma}
\begin{remark}\label{r:example}
We may use this lemma to identify the coordinate ring of $\mathcal X_{n,m}$ with the quantum cohomology ring $qH^*(Gr(k,n),\C)[q_k\inv]$ via Peterson's isomorphism from Theorem~\ref{t:Peterson}, where we recall that $k=n-m$. For example, 
\[
\sigma^{\ydiagram{1}}_{P_k}\qquad\leftrightarrow\qquad\mathfrak S^{\ydiagram{1}}_{P_k}(u^E(x_1,\dotsc, x_m))=\frac{E_{m-1}(\mathbf x)E_m^{k-1}(\mathbf x)}{E_m^{k}(\mathbf x)}=\frac{1}{x_1}+\dotsc+\frac{1}{x_m}.
\]
 \end{remark}
 We now describe the totally positive part of $X_{P_k}$.
\begin{definition}\label{d:zeta}
Let $\boldsymbol{\zeta}_{n,m}^{+} \in\mathcal X_{n,m}$  be the $m-$tuple of $n^{th}$ roots of $(-1)^{m+1}$ given by
\[
    \boldsymbol{\zeta}_{n,m}^{+} \coloneq  \{\zeta_{n}^{j} \,|\, j = 1,\dots,m\}, \qquad \text{with}\qquad  \zeta_{n}^j=\zeta_{n,m}^j = 
    \begin{cases}
      e^{\frac{2\pi i}{n}\left( l-j +1\right)} &: \quad m=2l+1 \\
      e^{-\frac{\pi i}{n}}e^{\frac{2\pi i}{n}\left( l-j+1 \right)}&: \quad m = 2l. 
\end{cases}
  \]
  Given $\boldsymbol{\zeta}_{n,m}^+$, let us write $z\boldsymbol{\zeta}_{n,m}^+\coloneq  \{z\zeta_{n}^{j} \,|\, j = 1,\dots,m\}$ for $z\in\C$. While the individual $\zeta_{n}^{j}$ do also depend on the choice of $m$, we will generally not include the $m$ in the notation, as it will be clear from context.
\end{definition}
\begin{remark}\label{r:rightmost}
  Informally speaking, if $m$ is odd then $\boldsymbol{\zeta}_{n,m}^{+}$ is the set of  ``right-most'' $n^{th}$ roots of unity. If $m$ is even, the set of $m$ right-most roots of $1$ is not well-defined, but we can take the right-most roots of $-1$ for $\boldsymbol{\zeta}_{n,m}^{+}$  instead, and this is again a unique set.  \end{remark}

\begin{theorem}\label{t:nonneg-toep-homeo}\cite[Theorem 8.4(2)]{rietsch2001grassmannians}
The map $ \C\setminus \{0\}\to  {X}_{P_k}(\C)$ defined by $z\mapsto u_n^{E}(z\boldsymbol{\zeta}_{n,m}^+)$, where $m=n-k$, restricts to  give a parametrisation
\[
    u_+=u_+^{(k,n)}\colon\,  \R_{> 0}\xrightarrow{\sim} X_{P_k}(\R_{>0})
 \]
of the totally positive part of $X_{P_k}$ (with analytic inverse).
\end{theorem}

\begin{remark}\label{r:rightmostpos}
Note that $\boldsymbol{\zeta}_{n,m}^{+}$ is, as a set, preserved by complex conjugation. Therefore the values $E_i(\boldsymbol{\zeta}_{n,m}^{+})$ (and the values $S_\lambda(\boldsymbol{\zeta}_{n,m}^{+})$ of all the  Schur polynomials) are automatically real,   
making it clear that the map in the proposition restricts to a map  $\R\setminus \{0\} \hookrightarrow  {X}_{P_k}(\R)$. While the $E_i(\boldsymbol{\zeta}_{n,m}^{+})$ are additionally positive, this does not follow for more general Schur polynomials  $S_\lambda(\boldsymbol{\zeta}_{n,m}^{+})$ unless  $\lambda\subseteq k\times m$. 
\end{remark}

\begin{lemma}\label{l:qviaparam} 
For an element $u\in X_{P_{k}}$ described as in Lemma~\ref{l:Toeplitz-polynomial} as  $u=u^E_n(x_1,\dotsc, x_m)$, we have that 
\[q(u)=(-1)^{m+1}\frac{1}{x_1^n}
\]
for the Peterson isomorphism from Theorem~\ref{t:Peterson}.
\end{lemma}
\begin{remark}
    Restricting to the totally positive part of $X_{P_k}$, the pullback of $q\in \C[X_{P_k}]$ along the parametrisation map $u_+$ from Theorem~\ref{t:nonneg-toep-homeo} is then given by $
   u_+^{*}(q) = \frac{1}{t^{n}}$  in terms of the parameter $t$ in $\R_{>0}$. We note that this special case of Lemma~\ref{l:qviaparam} follows already from $q^{k(n-k)} =\frac{1}{(\Delta_k)^n}$, which is  \eqref{e:qviadelta} in the Grassmannian setting.
\end{remark}
\begin{proof}[{Proof of Lemma~\ref{l:qviaparam}}]
Recall that $q=\sigma_{P_k}^{(1^k)}\star\sigma_{P_k}^{(m)}$, for example by \eqref{e:lambdakm} (where $k=n-m$).

Note that $(1^k)$ denotes the single-column Young diagram which corresponds to the permutation $s_1 s_2\dots s_k$. For the single-row Young diagram $(m)$ the corresponding permutation is $s_{n-1}\dotsc s_{k+1}s_k$.  Consider the functions on $X_{P_k}$  corresponding to $\sigma_{P_k}^{(1^k)}$ and $\sigma_{P_k}^{(m)}$. Namely, these are the following quotients of minors,
\begin{eqnarray*}
\mathfrak S_{P_k}^{(1^k)}(u)&=&\frac{\langle u\cdot v_{2}\wedge\dotsc\wedge v_{k+1},v_{n-k+1}\wedge\dotsc\wedge v_n\rangle}{\Delta_k(u)},\\
\mathfrak S_{P_k}^{(m)}(u)&=&\frac{\langle u\cdot v_{1}\wedge v_{2}\wedge\dotsc\wedge v_{k-1}\wedge v_n,v_{n-k+1}\wedge\dotsc\wedge v_n\rangle}{\Delta_k(u)}.
\end{eqnarray*}
Let us evaluate these functions on  $u^E=u^E(x_1,\dotsc, x_m)\in X_{P_k}$. Here $\{x_1,\dotsc, x_m\}\in\mathcal X_{n,m}$. 
The denominator, in either case, gives $E_m(x_1,\dotsc, x_m)^k$. The minor in the numerator of  $\mathfrak S_{P_k}^{(m)}(u)$ gives $E_m(x_1,\dotsc,x_m)^{k-1}$. The minor in the numerator of  $\mathfrak S^{(1^k)}_{P_k}(u)$ gives the Schur polynomial $S_\lambda$ associated to $\lambda=(m-1)\times k$ evaluated at $x_1,\dotsc, x_m$ (by the (dual) Jacobi-Trudi formula). All in all we have 
\begin{eqnarray*}
\mathfrak S_{P_k}^{(1^k)}(u^E)&=&\frac{S_{(m-1)\times k}(x_1,\dotsc, x_m)}{E_m(x_1,\dotsc, x_m)^k},\\
\mathfrak S_{P_k}^{(m)}(u^E)&=&\frac{1}{E_m(x_1,\dotsc, x_m)}.
\end{eqnarray*}
Now the set of $n$-th roots of $x_1^n$ consists of $x_1,\dotsc, x_m$ together with $k$ further distinct complex numbers, that we denote   $x_{m+1},\dotsc, x_n$. Thus we have constructed a well-defined element $[x_1,\dotsc, x_n]_{m,k}$ in $\mathcal Y_{m,k}$ from Definition~\ref{d:Ykm}. We  now note that $S_{(m-1)\times k}$ is a complicated Schur polynomial in $m$ variables compared to $S_{k\times(m-1)}$  (in $k$ variables), which is just a monomial. But we can apply Lemma~\ref{l:Ykm} to get
\begin{equation*}
S_{(m-1)\times k}(x_1,\dotsc, x_m)=S_{ k\times (m-1)}(-x_{m+1},\dotsc, -x_n)=E_k(-x_{m+1},\dotsc, -x_n)^{m-1}=(-1)^{k(m-1)}(x_{m+1}x_{m+2}\dotsc x_n)^{m-1}.
\end{equation*}
Thereby we obtain a Laurent monomial expression for $q(u^E)$ that we can simplify. Note that 
$x_1\dotsc x_n=(-1)^{n+1}x_1^n$ by the final part of Lemma~\ref{l:Ykm}.
\begin{multline*}
q(u^E)=\mathfrak S_{P_k}^{(1^k)}\mathfrak S_{P_k}^{(m)}(u^E)=\frac{(-1)^{k(m-1)}(x_{m+1}x_{m+2}\dotsc x_n)^{m-1}}{E_m(x_1,\dotsc, x_m)^{k+1}}=\frac{(-1)^{k(m-1)}(x_{m+1}x_{m+2}\dotsc x_n)^{m-1}}{(x_1\dotsc x_m)^{k+1}}\frac{(x_1\dotsc x_m)^{m-1}}{(x_1\dotsc x_m)^{m-1}}
\\ 
=\frac{(-1)^{k(m-1)}(x_{1}\dotsc x_n)^{m-1}}{(x_1\dotsc x_m)^{n}}=(-1)^{(n-m)(m-1)+(n+1)(m-1)}\frac{x_1^{n(m-1)}}{x_1^{nm}}
=(-1)^{(m+1)}\frac{1}{x_1^n}.
\end{multline*}
This is the formula that was to be proved.
\end{proof}
\section{Asymptotics of the Peterson Toeplitz variety for \texorpdfstring{$Gr(n-m,n)$}{Gr(n-m,n)} with fixed \texorpdfstring{$m$}{m}}\label{sec:asymptquantparam}

We now wish to describe the infinite totally nonnegative Toeplitz matrices that arise as limits of elements of the Peterson Toeplitz variety $X^{(n)}_{P_{n-m}}(\R_{>0})$ associated to $Gr(n-m,n)$ as $n\to\infty$ (with $m$ fixed).  
\begin{definition}
Let us write $\mathrm{Toep}_{\infty}(\C)$ for the infinite lower-triangular Toeplitz matrices over $\C$ of the form 
\begin{equation}\label{e:uinfty}
u^{\infty}=u_\infty(c_1,c_2,c_3,\dotsc)=\begin{pmatrix}
 1 &  & &  &        & \\
  c_{1} &1   &    &       & & \\
  c_{2} &c_{1} &  1&  &      &  \\
 \vdots &    \ddots &      \ddots  & \ddots &   &  \\
   &    &       &    \ddots  &  \ddots   &
\end{pmatrix},
\end{equation}
and $\mathrm{Toep}_n(\C)$, for the analogous $n\times n$ Toeplitz matrices $u_n(c_1,\dotsc, c_{n-1})$. We also write $\mathrm{Toep}_{\infty}(\R_{\ge 0})$  and $\mathrm{Toep}_{n}(\R_{\ge 0})$, respectively, for their totally nonnegative parts, where all minors are in $\R_{\ge 0}$. We define a truncation map $[\ ]_n:\mathrm{Toep}_{\infty}(\C)\to \mathrm{Toep}_{n}(\C)$ by setting
\begin{equation*}
{\left[ u^{\infty} \right]}_{n}=\begin{pmatrix}
 1 &  & &  &        & \\
  c_{1} &1   &    &       & & \\
  c_{2} &c_{1} &  \ddots &  &      &  \\
 \vdots &    \ddots &      \ddots  & \ddots &   &  \\
   c_{n-2}&    &       &    c_{1}  &  1    &  \\
  c_{n-1} & c_{n-2}    &     \dots  &   c_{2} &   c_{1}    &1
\end{pmatrix}.
\end{equation*}
Note that truncation restricts to 
\begin{equation}\label{e:trunc}
[\ ]_n:\mathrm{Toep}_{\infty}(\R_{\geq 0})\to \mathrm{Toep}_{n}(\R_{\geq 0}).
\end{equation}
Moreover, we have that $u^{\infty}$ is \emph{totally non-negative} if the truncations ${[u^{\infty}]}_{n}$ are totally non-negative for all $n$. We also use the notation $[u]_n$ for the $n\times n$ truncation of a (large enough) finite matrix.   

\end{definition}

\begin{definition}\label{d:convergence}
We say the sequence ${\left( u^{(n)} \right)}_{n>0}$ of finite totally nonnegative Toeplitz matrices $u^{(n)}\in\mathrm{Toep}_n(\R_{\ge 0})$ converges to $u^{\infty}\in \mathrm{Toep}_{\infty}(\R_{\ge 0})$ if the matrix entries converge individually.\end{definition}

\begin{remark}\label{r:convergence}
Note that $(u^{(n)})_n$ converges to $u^\infty$ 
if and only if $([u^{(n)}]_N)_{n>N}$ converges in $\Toep_N(\C)\cong\C^{n-1}$ for all $N\in\N$.  Equivalently, we could embed all $\Toep_{n}(\R_{\ge 0})$ as well as $\Toep_\infty(\R_{\ge 0})$ into $\R[[x]]$ sending $u$ to $p_u(x)$ and consider $\R[[x]]$ with the product topology (of coefficientwise convergence).
\end{remark}

Recall that we fix $m\in\N$ and consider $X_P^{(n)}=X_{P_{n-m}}^{(n)}$ for the Peterson Toeplitz variety associated to $Gr(n-m,n)$, that is, to the maximal parabolic subgroup $P=P_{n-m}$ in $SL_{n}$ (for all $n>m$). We may now consider sequences ${\left( u^{(n)} \right)}_{n>0}$ of Toeplitz matrices with $u^{(n)}\in X_{P_{n-m}}^{(n)}(\R_{>0})$, and ask when they converge to a totally nonnegative Toeplitz matrix, and also which totally nonnegative Toeplitz matrices arise as limits of such sequences.

\begin{remark}\label{r:Toepnnotinftp}While $X_{P_{n-m}}^{(n)}(\R_{>0})$ lies in $\Toep_n(\R_{\ge 0})$, the elements of   $X_{P_{n-m}}^{(n)}(\R_{>0})$ do not lie in the image of the truncation map \eqref{e:trunc} unless $m=1$. See also the related Remark~\ref{r:rightmostpos}. Indeed, if the natural infinite extension of $u^{(n)}=u_n(c_1,\dotsc,c_m)\in X_{P_{n-m}}(\R_{> 0})$, the Toeplitz matrix $u^\infty=u(c_1,\dotsc, c_m,0,\dotsc)$, compare \eqref{e:uinfty}, is totally nonnegative, then its generating function $1+c_1x+\dotsc+c_m x^m$ must have negative real roots. Namely by \cref{theorem:edrei-infinite}, these roots are given by $-\beta_1\inv,\dotsc,-\beta_m\inv\in\R_{<0}$ in terms of the Schoenberg parameters $\beta_1,\dotsc,\beta_m$. However, as long as $m>1$ we have that any $p_{u^{(n)}}(x)$ necessarily has some \textit{non-real} roots when $u^{(n)}\in X^{(n)}_{P_{n-m}}$, as follows from  Lemma~\ref{l:Toeplitz-polynomial}. 

Note that the Grassmannian setting differs significantly from the full flag variety $P=B$ one in this regard. In the full flag case we  have that the restrictions of any strictly totally positive infinite Toeplitz matrix $u^\infty\in \Toep_{\infty}(\R_{>0})$ will lie in  $X_{B}(\R_{>0})$ (although still not every element of $X_{B}(\R_{>0})$ will arise in this way). In any case, this means that the matrices $u^\infty\in \Toep_\infty(\R_{>0})$ automatically arise as  limits of sequences of elements $u^{(n)}\in X^{(n)}_{B_n}(\R_{>0})$, simply by setting $u^{(n)}:=[u^\infty]_n$.
\end{remark}

\begin{proposition}\label{prop:lim-q-param}
Let $u^{\infty}\in\Toep_\infty(\C)$ be the limit of a sequence ${(u^{(n)})}_{n}$, where $u^{(n)}\in X_{P_{n-m}}^{(n)}(\R_{> 0})$. Then $u^{\infty}\in \Toep_\infty(\R_{\ge 0})$, and its generating function is the polynomial
\[
 p_{u^{\infty}}(x)= 1+ c_1x + \cdots + c_m x^m = {(1+ \beta  x)}^{m},
\]
where $\beta\in \R_{\ge 0}$ is determined by 
\begin{equation}\label{e:qtobeta}
\beta =\lim\limits_{n\to \infty}\sqrt[n]{\frac{1}{q^{(n)}}}.
\end{equation}
Conversely, for any $\beta\in\R_{\ge 0}$ the infinite totally nonnegative Toeplitz matrix with generating function $ {(1+ \beta  x)}^{m}$ is obtained as the limit of a sequence $(u^{(n)})_n$ with $u^{(n)}\in X^{(n)}_{P_{n-m}}(\R_{>0})$. 
\end{proposition}

\begin{proof}
We have that $u^{(n)}$ is of the form $u_n(c^{(n)}_1,\dotsc, c^{(n)}_m,0,\dotsc, 0)$, assuming $n>m$. Therefore, if the sequence ${(u^{(n)})}_{n}$ converges to $u^{\infty}$, then we must have $u^{\infty}=u(c_1,\dots,c_m, 0,0,\dots)$ for some finite sequence $c_1,\dots,c_m\in \R_{\ge 0}$. In particular, the generating series of the infinite Toeplitz matrix $u^\infty$ is a polynomial, namely $p(x)=1+c_1x+\dots + c_m x^m$. Note that by  \cref{theorem:edrei-infinite} the matrix $u^{\infty}$ is in $\Toep_\infty(\R_{\ge 0})$ if and only if $p(x)=(1+\beta_1 x)\dotsc(1+\beta_mx)$ for Schoenberg parameters $\beta_1\ge \beta_2\ge\dotsc\ge \beta_m\ge 0$. Now, by the parametrization result Theorem~\ref{t:nonneg-toep-homeo} for $X_P^{(n)}$ we have $u^{(n)}=u_+^{(n-m,n)}(t_n)=u^{E}_n(t_n\boldsymbol\zeta^+_{n,n-m})$ for some   positive real parameter $t_{n}\in \R_{>0}$. Equivalently,  compare also Lemma~\ref{l:Toeplitz-polynomial}, we have
\begin{equation}\label{e:pun}
p_{u^{(n)}}(x) = 1+c^{(n)}_1 x+c^{(n)}_2 x^2+\dotsc +c^{(n)}_m x^m= \prod_{i=1}^m \left( 1 +  t_n \zeta^i_nx\right ), 
\end{equation}
where the $\zeta^i_n$ are the $n$-th roots of $(-1)^{m+1}$ from Definition~\ref{d:zeta} and $t_n>0$. Now $p(x)$ is the limit of the sequence ${(p_{u^{(n)}}(x))}_{n}$ in ${\R[x]}_{\leq m}$, and in particular $c_{m}^{(n)}$ converges to $c_m$. Therefore  $t_{n}= \sqrt[n]{\left(c^{(n)}_m\right)^k}$ converges to $\beta:=\sqrt[n]{(c_m)^k}$.  
By observing that $\lim\limits_{n\to \infty}\zeta_{n}^{i}=1$, independently of $i$, we see that the factors $\left( 1 +  t_n \zeta^i_nx\right )$ from \eqref{e:pun} all converge to  $\left( 1 +  \beta x\right )$. Therefore $u^\infty$ is totally nonnegative with Schoenberg parameters $\beta_i =\beta$, where $i=1,\dotsc, m$. Finally,  for the quantum parameter $q^{(n)}$ (as function on $X_{P_{n-m}^{(n)}}(\R_{>0})$) we have $q^{(n)}(u^{(n)})=\frac 1{t^n}$, see Lemma~\ref{l:qviaparam}. Therefore $\beta$ is given by $\lim\limits_{n\to \infty}\sqrt[n]{\frac{1}{q^{(n)}}}$. For the reverse implication, we can simply choose $t_n=\beta$ for all $n$ if $\beta>0$, or $t_n=\frac 1 n$ if $\beta=0$ and obtain a sequence $u^{(n)}$ that converges to the infinite Toeplitz matrix $u(\beta,\dots,\beta,0,\dotsc)$ with generating function $(1+\beta x)^m$. 
\end{proof}

\begin{remark}
    In Section~\ref{s:pflag-limits}, Proposition~\ref{p:pflag-q-lim-inf}, we will prove an analogous result in the case of certain partial flag varieties, and we will pose a conjecture for its generalisation, compatible with both limits as well as the result in the full flag case of~\cite{rietsch2025totallypositivetoeplitzmatrices}. 
\end{remark}

\subsection{Asymptotics of Schubert classes}\label{sec:asympt-schub-class}

Throughout this section, let $u^{\infty}$ be the limit of a sequence  ${(u^{(n)})}_{n}$ as in Section~\ref{sec:asymptquantparam}, where $u^{(n)}\in X_{P_{n-m}}^{(n)}(\R_{>0})$. By Proposition~\ref{prop:lim-q-param}, we have that  
\[u^{\infty} = u_{\infty}^{E}(\underbrace{\beta,\dots,\beta}_{m \text{ copies}},0,\dotsc),\quad\text{ for some $\beta\in \R_{\ge 0}$.}
\]
Equivalently, $u^{\infty} = u_{\infty}(E_1(\beta,\dots,\beta),\dots,E_m(\beta,\dots,\beta))$, where $E_i$ is the degree $i$ elementary symmetric polynomial in $m$ variables. Recall the functions $\mathfrak S^{w_\lambda}_{P_{n-m}}=\mathfrak{S}_{P_{n-m}}^{\lambda}$ corresponding to the quantum Schubert classes described in Definition~\ref{d:Swlambda}, see also \eqref{e:Slambda}. 

\begin{theorem}\label{thm:schub-class-lim}
 Suppose the sequence $(u^{(n)})_{n}$ with $u^{(n)}\in X_{P_{n-m}}^{(n)}(\R_{>0})$ converges to the infinite Toeplitz matrix $u^\infty$ with generating function $(1+\beta x)^m$ for $\beta>0$.  Let $\lambda$ be a fixed Young diagram with $\lambda_1\le m$. 
Then 
\[\lim_{n\to\infty}{\left( \mathfrak{S}^{\lambda}_{P_{n-m}}(u^{(n)}) \right)}=S_{\lambda'}\left( \frac{1}{\beta},\dots,\frac{1}{\beta} \right),\] 
where $S_{\lambda'}$ is the Schur polynomial in $m$ variables corresponding to the conjugate partition $\lambda'$ to $\lambda$. 
\end{theorem}

\begin{remark}
Clearly, $S_{\lambda'}\left( \frac{1}{\beta},\dots,\frac{1}{\beta} \right)=S_{\lambda'}\left( {1},\dotsc,{1} \right)\left(\frac{1}{\beta}\right)^{|\lambda'|}$, and $S_{\lambda'}\left( {1},\dotsc,{1} \right)$ is also the dimension of  the associated irreducible representation $V_{\lambda'}$ of $SL_m$ which can be given directly by the Weyl dimension formula. We have kept the expression in the form of a Schur polynomial value above to emphasise the analogy with Theorem~\ref{t:fullflag}. 
\end{remark}
Let us write $u=u^{(n)}=u_n(c_1,\dotsc, c_m)$ for now, as in \eqref{e:XPk}. Beginning with the simplest case, consider the Schubert class $\sigma^{\ydiagram{1}}_{P_k}$ and its associated function $\mathfrak S_{P_k}^{\ydiagram{1}}=\mathfrak{S}_{P_k}^{s_k}=\Delta^{s_k}_k(u)/{\Delta_k(u)}$, where we write $k$ for $n-m$ when convenient. Here $\Delta_k^{s_k}= \minor^{[n-k+1,n]}_{[k-1]\cap \left\{ k+1 \right\}}$, that is, the minor consisting of the last $k$ rows and the first $k-1$ columns  together with the $(k+1)^{st}$ column. For our $u$ 
this is a $k\times k$ minor of the form 
\begin{equation*}
  \left|\begin{pmatrix}
    c_m & c_{m-1}&\cdots & \\
        &\ddots &\ddots & \vdots\\
        & &c_{m} &c_{m-2} \\
    &\cdots &0 &c_{m-1}
\end{pmatrix}\right|=  c_{m-1}c_m^{k}.
\end{equation*}
Hence, since $\Delta_k(u) = c_m^k$, we have $\mathfrak{S}^{s_{k}}(u) = \frac{c_{m-1}}{c_m}$.

For another example, consider the permutation $s_{k-1}s_k$. We have $D^{s_{k-1}s_k}_k = \minor^{[n-k+1,n]}_{[k-2]\cap [k,k+1]}$. Similarly, we have
\begin{equation*}
 \left|\begin{pmatrix}
    c_m &c_{m-1} &\cdots & & \\
        &\ddots &\ddots & & \vdots \\
        & &c_{m}  &c_{m-1} &c_{m-2} \\
        & &0 & c_{m-1}&c_{m-2}\\
        &\cdots &0 & c_{m}&c_{m-1}
\end{pmatrix}\right|=  c_{m-1}^{2}c_m^{k-2} - c_{m-2}c_m^{k-1}. 
\end{equation*}
Hence $\mathfrak{S}^{s_{k-1}s_{k}}(u)= \frac{c_{m-1}^2}{c_m^2}-\frac{c_{m-2}}{c_m}$. A key observation is that in terms of the $c_i$ these functions do not depend on $n$ if $m$ is fixed. (They would depend on $n$ if we fixed instead $k$.) 

\begin{lemma}\label{lem:schur-k-independent}
 Fix $m$ and let $n>m$, setting $k=n-m$. Then $\mathfrak{S}_{P_{k}}^{s_{k-l} \cdots s_{k}}(u^{(n)})$ is a polynomial in the entries $c_1,\dotsc, c_m$ of $u^{(n)}$ that is stable, in the sense that it does not depend on $n$. 
\end{lemma}

\begin{proof}
 In general, $\Delta_k^{s_{k-l}\cdots s_k}= \minor^{[n-k+1,n]}_{[k-1-1]\cap [k-l+1,k+1]}$, i.e.\ the minor consisting of the last $k$ rows and the first $k+1$ columns with the $l^{th}$ column removed. We can write $\Delta_k^{s_{k-l}\cdots s_k}$ as the determinant of the block matrix $\begin{pmatrix}A&B\\ 0 & C\end{pmatrix}$, where $A$ and $C$ are $(k-l-1)\times( k-l-1)$ and $(l+1)\times( l+1)$ matrices respectively given by \[
   A =  \begin{pmatrix}
     c_{m} &c_{m-1} &\cdots & \\
           & c_{m}&\ddots & \vdots\\
           & & \ddots&c_{m-1} \\
         0  & & & c_{m}
\end{pmatrix}, \qquad C= \begin{pmatrix}
 c_{m-1} &c_{m-2} &\cdots & \\
      c_{m}     & c_{m-1}&& \vdots\\
           &\ddots\ \quad \ \  & \ddots\qquad\ \ & \vdots\\
          &\qquad\ddots&\quad \ddots&c_{m-2}\\
         0  & &c_{m}\quad & c_{m-1}
\end{pmatrix}.
 \]
 Since $\det(A) = {(c_m)}^{k-l-1}$, we have \[
  \mathfrak{S}^{s_{k-l}\cdots s_k}(u)= \frac{\det(A)\det(C)}{\Delta^{k}}=\frac{ {(c_m)}^{k-l-1}\det(C)}{{(c_m)}^{k}} = \frac{\det(C)}{{(c_m)}^{l+1}},
\]
which does not depend on $k$ (or equivalently, $n$).
\end{proof}

We now write $u^{(n)}=u_n(c_1^{(n)},\dotsc, c_m^{(n)})$ and make the assumption that these matrices converge to $u^\infty=u^\infty(c_1,\dotsc, c_m)$. Recall that we have by Proposition~\ref{prop:lim-q-param} that  $c_{m-r} = \binom{m}{r}\beta^{m-r}=E_{m-r}(\underbrace{\beta,\dotsc, \beta}_m)$ for some $\beta\in\R_{\ge 0}$.

\begin{lemma}\label{lem:one-strip-schub-lim}
 Given a converging sequence with $u^{(n)} \to  u^{\infty}$, for $w =s_{n-m-l}\cdots s_{n-m}$ we have  $\lim\limits_{n\to \infty}(\mathfrak{S}^{w}(u^{(n)})) = H_{l+1}(\frac{1}{\beta}, \dots,\frac{1}{\beta})$, where $H_{l+1}$ is the degree $l+1$ complete elementary symmetric polynomial in $m$ variables. 
\end{lemma}

\begin{proof}
  Let $u^{(n)} = u(c_1^{(n)},\dots, c_m^{(n)})$, and define, similarly to  Lemma~\ref{lem:schur-k-independent}, the determinants \[
   \mathsf{A}_l(u^{(n)})\coloneq \det \underbrace{
    \begin{pmatrix}
c^{(n)}_{m} &c^{(n)}_{m-1} &\cdots & \\
        & c^{(n)}_{m}&\ddots & \vdots\\
           & & \ddots&c^{(n)}_{m-1} \\
         0  & & & c^{(n)}_{m}
\end{pmatrix}
}_{(l+1)\times(l+1)}\qquad \mathsf{C}_l(u^{(n)})\coloneq  \det \underbrace{
    \begin{pmatrix}
c^{(n)}_{m-1} &c^{(n)}_{m-2} &\cdots & \\
      c^{(n)}_{m}     & c^{(n)}_{m-1}&\ddots & \vdots\\
           &\ddots & \ddots&c^{(n)}_{m-2} \\
         0  & &c^{(n)}_{m} & c^{(n)}_{m-1}
\end{pmatrix}.
}_{(l+1)\times(l+1)}
  \]

  It follows that $\lim\limits_{n\to \infty}\mathfrak{S}^{w}(u^{(n)}) = \frac{\mathsf{C}_l(u^{\infty})}{\mathsf{A}_l(u^{\infty})}$. Moreover, $c_i^{(n)}$ converges to $E_i(\beta,\dots,\beta)$, which we shorthand with $E_i(\beta^{(m)})$,  so \[
    \mathsf{C}_l(u^{\infty})=  \det 
    \begin{pmatrix}
E_{m-1} (\beta^{(m)})&E_{m-2}(\beta^{(m)}) &\cdots & \\
      E_{m}(\beta^{(m)})     & E_{m-1}(\beta^{(m)})&\ddots & \vdots\\
           &\ddots & \ddots&E_{m-2}(\beta^{(m)}) \\
         0  & &E_{m}(\beta^{(m)}) & E_{m-1}(\beta^{(m)})
\end{pmatrix}
  \]

  Noticing that the degrees of the $E_{i}$s are decreasing by one going up the columns, and that the entries are $0$ below the $E_m$ entry in each column, we see that this is the transpose of the Jacobi-Trudi determinant formula for the Schur polynomial $S_{\lambda_1}(
  \beta,\dots,\beta
  )=S_{\lambda_{(1)}}(\beta^{(m)})$ associated to the Young diagram $\lambda_{(1)}$, with $m-1$ rows of length $l+1$. By a similar argument, $\mathsf{A}_l(u^{\infty})$ is the Schur polynomial for a Young diagram $\lambda_{(2)}$ with $m$ rows of length $l+1$. A well-known property of the Schur polynomial (in $m$ variables) associated to a Young diagram $\lambda$ is that the degrees of the terms are the total number of boxes, and they are indexed by the semi-standard Young tableau of shape $\lambda$. That is, a filling of $\lambda$ with entries in the set $\left\{ 1,\dots,m \right\}$ so that entries increase weakly across rows and strictly down columns. Let $\mathsf{SSYT} (\lambda)$ denote the set of all SSYTs of shape $\lambda$.

  Recall that $S_{\underbrace{\ydiagram{2}\cdots\ydiagram{1}}_{l+1}} = H_{l+1}$. We now claim that 
  \begin{equation}\label{eqn:shur-inv-equality}
 \frac{S_{\lambda_{(1)}}(\beta,\dots,\beta)}{S_{\lambda_{(2)}}(\beta,\dots,\beta)} = S_{\underbrace{\ydiagram{2}\cdots\ydiagram{1}}_{l+1}}\left(\frac{1}{\beta},\dots,\frac{1}{\beta}\right).
\end{equation}

  Since all the variables of each Schur polynomial are set to $\beta$, we have
  \begin{align*}
    \frac{S_{\lambda_{(1)}}(\beta,\dots,\beta)}{S_{\lambda_{(2)}}(\beta,\dots,\beta)} & = \frac{\sum\limits_{T\in \mathsf{SSYT}(\lambda_1)}\beta^{(m-1)(l+1)}}{\sum\limits_{T\in \mathsf{SSYT}(\lambda_2)}\beta^{(m)(l+1)}},\\
     &= \frac{1}{\beta^{m(l+1)}}\sum\limits_{T\in \mathsf{SSYT}(\lambda_1)}\beta^{(m-1)(l+1)}, \\
    &= \sum\limits_{T\in \mathsf{SSYT}(\lambda_1)}\frac{1}{\beta^{(l+1)}}.
\end{align*}
The second equality is due to the fact that $\left|  \mathsf{SSYT}(\lambda_{(2)}) \right| = 1$.  We now claim that $\left|  \mathsf{SSYT}(\lambda_{(1)}) \right|= |  \mathsf{SSYT}(\underbrace{\ydiagram{2}\cdots\ydiagram{1}}_{l+1}) |$. An element of $\mathsf{SSYT}_{\lambda_{(1)}}$ is a strictly decreasing  filling of $l+1$ columns of length $m-1$, so is uniquely determined by which element of $\left\{ 1,\dots,m \right\}$ is missing from the column. In order for the rows to be weakly increasing to the right, the missing element in a column must be less than or equal to the missing element from the previous column to the left. Hence we get a weakly decreasing sequence of numbers in $[m]$,  whose reverse gives a SSYT for the strip of $l+1$ boxes. Hence we have  a set-theoretic bijection between $\mathsf{SSYT}(\lambda)$ and $\mathsf{SSYT}(\underbrace{\ydiagram{2}\cdots\ydiagram{1}}_{l+1})$. The equality~\eqref{eqn:shur-inv-equality} follows. 
\end{proof}

\begin{proof}[Proof of Theorem~\ref{thm:schub-class-lim}]
  The case of a Young diagram  with one column is given by Lemma~\ref{lem:one-strip-schub-lim}, since the permutation $s_{n-m-\ell}\cdots s_{n-m}$ corresponds to the partition $(1^{\ell+1})={(1,\dotsc,1)}$ with $\ell+1$ parts. We use the quantum Giambelli formula due to Bertram~\cite{BERTRAM1997289}. Let $\lambda = (\lambda_1,\dots,\lambda_r )$ be a Young diagram with $\lambda\subseteq k\times m$, and write $\lambda'= ( \lambda'_1,\dots,\lambda'_j )$ for the conjugate partition. Let us furthermore write $\sigma^{(\lambda_i')^t}$ for the quantum Schubert class $\sigma^{(1,\dotsc,1)}$ associated to the Young diagram $(1^{\lambda_i'})$ with $\lambda_i'$ parts. Then by the (dual) quantum Giambelli formula we have

    \begin{equation}
  \label{eq:q-giambelli}
   \sigma^{\lambda} = \det 
    \begin{pmatrix}
      \sigma^{(\lambda'_{1})^t} &  \sigma^{(\lambda'_{1} + 1)^t} & \sigma^{(\lambda'_{1}+2)^t} &\cdots & \sigma^{(\lambda'_{1} + r-1)^t} \\
                          \sigma^{(\lambda'_{2}-1)^t}  & \sigma^{(\lambda'_{2})^t} & & & \\
                         \sigma^{(\lambda'_{3}-2)^t}   &   \sigma^{(\lambda'_{3}-1)^t}& \sigma^{(\lambda'_{3})^t}& & \\
                          \vdots &  & &\ddots &\vdots \\
                       \sigma^{(\lambda'_{r}- r+1)^t}     &  &\cdots & & \sigma^{(\lambda'_{r})^t}
\end{pmatrix},
  \end{equation}
and hence we have a sequence of determinants given by \[
  \mathfrak{S}^{\lambda}(u^{(n)}) = \det 
    \begin{pmatrix}
       \mathfrak{S}^{(\lambda'_{1})^t}(u^{(n)}) &   \mathfrak{S}^{(\lambda'_{1} + 1)^t}(u^{(n)}) &  \mathfrak{S}^{(\lambda'_{1}+2)^t}(u^{(n)}) &\cdots &  \mathfrak{S}^{(\lambda'_{1} + r-1)^t}(u^{(n)}) \\
                           \mathfrak{S}^{(\lambda'_{2}-1)^t} (u^{(n)}) &  \mathfrak{S}^{(\lambda'_{2})^t}(u^{(n)}) & & & \\
                          \mathfrak{S}^{(\lambda'_{3}-2)^t}(u^{(n)})   &    \mathfrak{S}^{(\lambda'_{3}-1)^t}(u^{(n)})&  \mathfrak{S}^{(\lambda'_{3})^t}(u^{(n)})& & \\
                          \vdots &  & &\ddots &\vdots \\
                        \mathfrak{S}^{(\lambda'_{r}- r+1)^t}  (u^{(n)})   &  &\cdots & &  \mathfrak{S}^{(\lambda'_{r})^t}(u^{(n)}), 
\end{pmatrix}
\]
where $\mathfrak{S}^{(\ell)^t}(u^{(n)}) = \mathfrak{S}^{s_{n-m-(\ell-1)}\cdots s_{n-m}}(u^{(n)})$. Applying Lemma~\ref{lem:one-strip-schub-lim} it follows that
\begin{align*}
 \lim \limits_{n\to \infty}\mathfrak{S}^{\lambda}(u^{(n)}) = \det \begin{pmatrix}
      H_{\lambda_{1}'}(\frac{1}{\beta}, \dots,\frac{1}{\beta}) &  H_{\lambda_{1}' + 1}(\frac{1}{\beta}, \dots,\frac{1}{\beta}) & H_{\lambda_{1}'+2}(\frac{1}{\beta}, \dots,\frac{1}{\beta}) &\cdots & H_{\lambda_{1}'+r-1}(\frac{1}{\beta}, \dots,\frac{1}{\beta}) \\
                         H_{\lambda_{2}'-1}(\frac{1}{\beta}, \dots,\frac{1}{\beta}) & H_{\lambda_{2}'} (\frac{1}{\beta}, \dots,\frac{1}{\beta})& & & \\
                         H_{\lambda_{3}'-2} (\frac{1}{\beta}, \dots,\frac{1}{\beta})  &   H_{\lambda_{3}'-1}(\frac{1}{\beta}, \dots,\frac{1}{\beta})& H_{\lambda_{3}'}(\frac{1}{\beta}, \dots,\frac{1}{\beta}) & & \\
                          \vdots &  & &\ddots &\vdots \\
                       H_{\lambda_{r}'-r+1}  (\frac{1}{\beta}, \dots,\frac{1}{\beta})   &  &\cdots & & H_{\lambda_{r}'}(\frac{1}{\beta}, \dots,\frac{1}{\beta}).
\end{pmatrix}
\end{align*}
This is just the  Jacobi-Trudi formula expressing the Schur polynomial $S_{\lambda'}$ in terms of complete symmetric polynomials. So we can conclude that \[
   \lim \limits_{n\to \infty}\mathfrak{S}^{\lambda}(u^{(n)}) = S_{\lambda'}\left(\frac{1}{\beta}, \dots,\frac{1}{\beta}\right). 
\]
\end{proof}

\section{Asymptotics of the Peterson Toeplitz variety for $Gr(k,n)$ with fixed $k$}\label{sec:asymptquantparam2}

 Let $X_{P_{k}}^{(n)}$ denote the Peterson Toeplitz variety associated to $Gr(k,n)$, where now we will fix $k$. In this section, we show the following proposition, which will follow from  Proposition~\ref{prop:lim-q-param}  by applying a natural symmetry operation related to the isomorphism $Gr(n-m,n)\cong Gr(m,n)$. 
\begin{proposition}\label{prop:lim-q-param2}
Let $u^{\infty}\in\Toep_\infty(\C)$ be the limit of a sequence ${(u^{(n)})}_{n}$, where $u^{(n)}\in X_{P_{k}}^{(n)}(\R_{> 0})$. Then $u^{\infty}$ lies in $\Toep_\infty(\R_{\ge 0})$ and its generating function is given by
\begin{equation}\label{e:palpha}
 p_{u^{\infty}}(x)= \frac{1} {(1- \alpha  x)^{k}},
\end{equation}
where $\alpha$ is determined by 
\begin{equation}\label{e:qalpha}\alpha =\lim\limits_{n\to \infty}\sqrt[n]{\frac{1}{q^{(n)}(u^{(n)})}}.
\end{equation} Conversely, for any $\alpha\in\R_{\ge 0}$ the infinite totally nonnegative Toeplitz matrix with generating function $\frac{1} {(1- \alpha  x)^{k}}$ is obtained as the limit of a sequence $(u^{(n)})_n$ with $u^{(n)}\in X^{(n)}_{P_{k}}(\R_{>0})$. 
\end{proposition}

We first prove a lemma that constructs the isomorphism $qH^*(Gr(k,n))\cong qH^*(Gr(m,n))$ that swaps $\sigma_{P_k}^\lambda$ and $\sigma_{P_m}^{\lambda'}$, see Section~\ref{s:qcoh}, on the level of Peterson Toeplitz varieties. Recall Definitions~\ref{d:Ykm} and \ref{d:Xnm}. 

\begin{lemma}\label{l:symm} For $k+m=n$, the isomorphism $\Psi=\Psi_{m,k}: X^{(n)}_{P_{m}}\longrightarrow  X^{(n)}_{P_k}$ determined by
\begin{eqnarray*}
    \Psi^*: \C[X^{(n)}_{P_{k}}]&\longrightarrow & \C[X^{(n)}_{P_m}]\\
    \mathfrak S^\lambda_{P_k}&\mapsto &\mathfrak S^{\lambda'}_{P_m}
\end{eqnarray*} 
can be described in the following equivalent ways. 
\begin{enumerate}
\item For $u\in X_{P_m}$, we have
$\Psi(u)=\bar w_0\inv (u\inv)^T \bar w_0$.
\item Writing $u\in X_{P_m}$ as $u_n(c_1,\dotsc, c_k)$ in terms of its entries, we have
\begin{eqnarray*}\label{e:symm}
\Psi(u_{n}(c_1,\dotsc, c_k))&= &
u_{n}(-c_1,c_2,\dotsc, (-1)^k c_k)\inv,
\end{eqnarray*}
\item Writing $u$ as $u_{n}^E(x_1,\dotsc, x_k)$ for $(x_1,\dotsc, x_k)\in\mathcal X_{n,k}$, we have 
\begin{eqnarray*}\label{e:symm}
\Psi(u_{n}^E(x_1,\dotsc, x_k))&= &
u_{n}^E(-x_1,-x_2,\dotsc, - x_k)\inv.
\end{eqnarray*}
\item 
Writing $u=u_n^E(x_1,\dotsc, x_k)$ for $[x_1,\dotsc, x_k,x_{k+1},\dotsc,x_n]_{k,m}\in\mathcal Y_{k,m}$, we have
\begin{eqnarray*}\label{e:symm}
\Psi(u_{n}^E(x_1,\dotsc, x_k))&= &
u_{n}^E(-x_{k+1},-x_{k+2},\dotsc,-x_n).
\end{eqnarray*}
\end{enumerate}
Furthermore, we have that $\Psi$ restricts to a bijection $\Psi_{>0}: X^{(n)}_{P_{m}}(\R_{>0})\longrightarrow X^{(n)}_{P_k}(\R_{>0})$ of  the totally positive parts, where it satisfies
\begin{equation}\label{e:Psipos}
\Psi\left(u_+^{(m,n)}(t)\right)=u_+^{(k,n)}(t)    
\end{equation}
for $t\in \R_{>0}$.
\end{lemma}
\begin{remark}\label{r:qsymm} We note that  $q_k(\Psi(u))=q_m(u)$ follows automatically, since the isomorphism $qH^*(Gr(k,n))\cong qH^*(Gr(m,n))$ from  \cite{BCFF}, which is determined by sending $\sigma^\lambda_{P_k}$ to $\sigma^{\lambda'}_{P_m}$, also sends $q_k$ to $q_m$. Also note that one further description of $\Psi$ is given by $\Psi(u)=\iota(u)$, by Remark~\ref{r:iotaequiv}.
\end{remark}
\begin{proof} Note that $\bar w_0\inv(u\inv)^T\bar w_0=\iota(u)$ since $u$ is Toeplitz. The formula for $\Psi$ given in (1) follows from Lemma~\ref{l:lambda'}. The description in (2) follows from  the one in (1). Namely, (1) can be rewritten as $\Psi(u)=(\bar w_0\inv u^T \bar w_0)\inv$, but transposing and then conjugating by $\bar w_0$ the Toeplitz matrix $u_n(c_1,\dotsc, c_k)$ has precisely the effect of alternatingly reversing the signs to give $u_n(-c_1,c_2,\dotsc, (-1)^k c_k)$.  The description of $\Psi$ from (3) is straightforwardly equivalent to the one in (2).  
Finally, it is clear that
\[
u_n^{E}(-x_1,\dotsc, -x_k)\inv=u_n(H_1(x_1,\dotsc, x_k),\dotsc, H_m(x_1,\dotsc, x_k)),
\]
using also that $ H_{m+1}(x_1,\dotsc, x_k)=\dotsc=H_n(x_1,\dotsc,x_k)=0$ for $(x_1,\dotsc,x_k)\in \mathcal X_{n,k}$. Now the description in (4) follows, since  $H_j(x_1,\dotsc, x_k)=E_j(-x_{k+1},\dotsc, -x_n)$, see Lemma~\ref{l:Ykm}. Finally, if $\{x_1,\dotsc, x_k\}=\boldsymbol{\zeta}^+_{n,k}$ then it follows that $\{-x_{k+1},\dotsc, -x_n\}=\boldsymbol{\zeta}^+_{n,m}$, so that (4) also implies \eqref{e:Psipos}.
\end{proof}

\begin{proof}[{Proof of Proposition~\ref{prop:lim-q-param2}}] 
Using Definition~\ref{d:zeta} and Theorem~\ref{t:nonneg-toep-homeo}, we may write 
\begin{equation}
    u^{(n)}=u^{(k,n)}_+(t_n)=u^E(t_n\zeta^1_{n,n-k},\dotsc, t_n\zeta^m_{n,m}),
\end{equation} 
for some $t_n\in\R_{>0}$. On the other hand, by Lemma~\ref{l:symm} we have 
\[
u^{(k,n)}_+(t_n)\overset{{\eqref{e:Psipos}}}=\Psi((u^{(n-k,n)})(t_n))\overset{\ref{l:symm}.(3)}= u_n^E(-t_n\zeta^1_{n,k},\dotsc,-t_n\zeta^{k}_{n,k})\inv.
\]
For the generating polynomial we therefore have
\begin{equation}
p_{u^{(n)}}(x)\equiv\left(1-E_1(t_n\boldsymbol{\zeta}^+_{n,k})x+E_2(t_n\boldsymbol{\zeta}^+_{n,k})x^2 \pm\dotsc +(-1)^{k}E_{k}(t_n{\boldsymbol{\zeta}}^+_{n,k})x^{k}\right)\inv \equiv\frac{1}{\prod_{j=1}^{k}(1-\zeta^j_{n,k}t_nx)} \mod x^n,
\end{equation}
which identifies $p_{u^{(n)}}(x)$ as the truncation (at $x^{n-1}$) of the infinite power series on the right-hand side. 

Now, since $k$ is fixed and the $\zeta^{j}_{n,k}$ converge to $1$ as $n\to \infty$, we have that $p_{u^{(n)}}(x)$ converges if and only if the $t_n$ converge. Let us set $\alpha:=\lim_{n\to\infty}{t_n}$, noting that $\alpha\ge 0$. The limit of the $p_{u^{(n)}}(x)$ in $\R[[x]]$ (see Remark~\ref{r:convergence}) is then given by
\[
\lim_{n\to\infty} p_{u^{(n)}}(x)=\frac{1}{(1-\alpha x)^k},
\]
since the truncation disappears in the infinite limit. Furthermore, since $q(u_+(t))=\frac{1}{t^n}$, see Lemma~\ref{l:qviaparam}, we obtain $\alpha$ from $(q^{(n)}(u^{(n)}))_n$ by the limit formula given in \eqref{e:qalpha}. For the final statement of Proposition~\ref{prop:lim-q-param2}, it suffices to take the sequence given by $u^{(n)}=u^{(k,n)}_+(\alpha)$ if $\alpha>0$ or $u^{(n)}=u^{(k,n)}_+(\frac{1}n)$ if $\alpha=0$. 
 \end{proof}
Finally, we have the analogue to~\cref{thm:schub-class-lim} about asymptotics of Schubert classes for the case where $k$ remains constant, which now follows easily from the results already proved. Note that given a partition $\lambda$ with at most $k$ parts, it will give rise to a function $\mathfrak S^\lambda_{P_k}$ (or written $\mathfrak S^{\lambda,(n)}_{P_k}$) on $X^{(n)}_{P_k}$ for all $n>k$.
\begin{theorem}\label{thm:schub-class-lim-k}
 Suppose the sequence $(u^{(n)})_{n}$ with $u^{(n)}\in X_{P_k}^{(n)}(\R_{>0})$ converges to the infinite Toeplitz matrix $u^\infty$ with generating function $\frac 1{(1-\alpha x)^k}$, where $\alpha\in\R_{>0}$. Let $\lambda=(\lambda_1,\dotsc, \lambda_k)$ be a partition with at most $k$ parts. 
 
 Then 
 \[
 \lim_{n\to\infty}{\left( \mathfrak{S}^{\lambda}_{P_k}(u^{(n)}) \right)}=S_{\lambda}\left( \frac{1}{\alpha},\dots,\frac{1}{\alpha} \right),\] 
 where $S_{\lambda}$ is the Schur polynomial in $k$ variables associated to $\lambda$. 
\end{theorem}

\begin{proof} As in the proof of Proposition~\ref{prop:lim-q-param2}, we have that $u^{(n)}=u_+^{(k,n)}(t_n)$ where $t_n\in\R_{>0}$, and $\lim_{n\to\infty}(t_n)=\alpha$. Moreover, by Lemma~\ref{l:symm} we have 
$u^{(k,n)}_+(t_n)=\Psi(u^{(n-k,n)}_+(t_n))$, and furthermore
\[
\mathfrak{S}^{\lambda}_{P_k}(u^{(n)})= \mathfrak{S}^{\lambda}_{P_k}\left(u^{(k,n)}_+(t_n)\right)=\mathfrak{S}^{\lambda}_{P_{k}}\left(\Psi\left(u^{(n-k,n)}_+(t_n)\right) \right)
=\mathfrak S_{P_{n-k}}^{\lambda'}\left(u^{(n-k,n)}_+(t_n)\right).\]
Now $u^{(n-k,n)}_+(t_n)$ converges to the infinite Toeplitz matrix with generating function $(1+\alpha x)^k$ as in the proof of Proposition~\ref{prop:lim-q-param}, and therefore Theorem~\ref{thm:schub-class-lim} applies (with $m$ replaced by $k$, $\beta$ by $\alpha$, and $\lambda$ by $\lambda'$), giving 
\begin{equation*}
\lim_{n\to\infty}\mathfrak S_{P_{n-k}}^{\lambda'}\left(u^{(n-k,n)}_+(t_n)\right)=S_\lambda\left(\frac 1\alpha,\dotsc,\frac 1\alpha\right).
\end{equation*}
This implies the desired formula for the limit of the $\mathfrak{S}^{\lambda}_{P_k}(u^{(n)})$.
\end{proof}

\section{Asymptotics of the Peterson Toeplitz variety for $Gr(n,2n)$}\label{sec:gr2n}
\begin{proposition}\label{p:exponential} For $\eta\in \R_{\ge 0}$ let $u^\infty(\eta)$ be the infinite totally nonnegative Toeplitz matrix 
with generating function $p_{u^\infty}(x)=e^{\eta x}$. Then there exists a sequence $(u^{(2n)})_n$ where $u^{(2n)}\in X^{(2n)}_{P_n}(\R_{>0})$ that converges to $u^\infty(\eta)$. Moreover,  if a sequence $(u^{(2n)})_n$ with $u^{(2n)}\in X^{(2n)}_{P_n}(\R_{>0})$ converges to an infinite Toeplitz matrix, then we must have 
\[
\lim_{n\to\infty}(u^{(2n)})=u^\infty(\eta)=u_\infty\left(\eta,\frac{1}{2!}\eta^2,\frac{1}{3!}\eta^3,\frac{1}{4!}\eta^4,\dotsc\right)
\]
for some $\eta\ge 0$.
\end{proposition}
\begin{proof}
Let us write $u^{(2n)}=u_+^{(n,2n)}(t_n)=u_{(2n)}(t_n\boldsymbol{\zeta}^+_{2n,n})$ where $t_n\in \R_{>0}$, using Theorem~\ref{t:nonneg-toep-homeo}. The generating polynomial associated to $u^{(2n)}$ is
\begin{equation}\label{e:n2ngen}
p_{u^{(2n)}}(x)=\prod_{j=1}^n(1+t_n\zeta^j_{2n}x).
\end{equation} 
Letting $t_n=\frac \gamma n$, we compute the limit of $u^{(2n)}$ as $n\to \infty$ explicitly.

Set $v\coloneq e^{\frac{2\pi i}{2n}}=e^{\frac{\pi i}{n}}$ and note that 
\[ \zeta_{2n}^j = v^{\frac{n-2j+1}2}\quad\text{ for $j=1,\dotsc, n$.}
    \]
As in \cite{rietsch2001grassmannians,Karlin}, this allows us to compute $p_{u^{(2n)}}(x)$ directly using a version of the Gaussian $q$-binomial formula,
\begin{equation}\label{e:pu2n}
\prod_{j=1}^n(1+v^{-\frac{n+1}2+j}x)=\sum_{j=0}^{n}\qbinom{n}{i}_v x^j,
\end{equation}
where
\begin{equation*}
    \qquad [k]_v=\frac{v^{\frac{-k}2}-v^{\frac{k}2}}{v^{-\frac{1}{2}}-v^{\frac{1}{2}}},\quad [k]_v! =[1]_v[2]_v\dotsc[k]_v,\quad \text{and} \quad\qbinom{n}{j}_v=\frac{[n]_v!}{[j]_v![n-j]_v!}.
\end{equation*}
Note that 
\begin{equation}
[k]_v=\frac{v^{\frac{-k}2}-v^{\frac{k}2}}{v^{-\frac{1}{2}}-v^{\frac{1}{2}}}=\frac{\sin(\frac{k\pi}{2n})}{\sin(\frac{\pi}{2n})},
\end{equation}
so that the coefficient of $x^\ell$ in $p_{u^{(2n)}}(x)$ is given by
\begin{equation}\label{e:ekzeta}
\frac{\gamma^\ell}{n^\ell}E_\ell(\boldsymbol{\zeta}^+_{2n,n})=\frac{\gamma^\ell}{n^\ell}\frac{\prod_{j=1}^n\sin(\frac{j\pi}{2n})}{\left(\prod_{j=1}^{\ell}\sin(\frac{j\pi}{2n})\right)\left(\prod_{j=1}^{n-\ell}\sin(\frac{j\pi}{2n})\right)}=\gamma^\ell\frac{\prod_{r=n-\ell+1}^n\sin(\frac{r\pi}{2n})}{\prod_{j=1}^{\ell}(n\sin(\frac{j\pi}{2n}))}.
\end{equation}
Here we have a fixed number $\ell$ of factors in the numerator and denominator in the right-hand side expression. The factors in the denominator  are of the form $n\sin(\frac {j\pi}{2n})$, where $x=\frac {j\pi}{2n}$ is tending to $0$ as $n\to\infty$. Therefore, we can make  use of the linear approximation $\sin(x)\sim  x$ to show that $\lim_{n\to\infty}(n\sin(\frac {j\pi
}{2n}))= \frac {j\pi}2$, as higher-degree terms vanish in the limit. Meanwhile, the corresponding $\ell$ factors in the numerator are each approaching $\sin(\frac{\pi}2)=1$ as $n\to\infty$, so that overall we find
\[
\frac{\gamma^\ell}{n^\ell}E_\ell(\boldsymbol{\zeta}^+_{2n,n})\to \gamma^\ell\frac{2^\ell}{\ell!\pi^\ell}\qquad (n\to\infty).
\]
Finally, the polynomials $p_{u^{(n)}}(x)$ converge coefficientwise to 
\[
\sum_{\ell=1}^\infty \frac{1}{\ell!}\left(\frac{2\gamma}{\pi}\right)^\ell x^\ell=e^{\frac{2\gamma}{\pi}x}.
\]
We obtain $e^{\eta x}$ by setting $\gamma=\frac{\pi}2\eta$.

For the converse, suppose we have an arbitrary  sequence $(u^{(2n)})_n$ with $u^{(2n)}\in X_{P_n}^{(2n)}(\R_{>0})$, converging to $u_{\infty}(c_1, c_{2},\dots)$. By Theorem~\ref{t:nonneg-toep-homeo}, we can write $u^{(2n)} = u_{2n}^E(t_n \boldsymbol{\zeta}_{2n,n}^{+}))$, with $t_{n}\in \R_{>0}$.   The generating polynomial for each term is given by \[
  p_{u^{(2n)}}(x) = \prod\limits_{i=1}^n \left( 1 + t_n\zeta_{2n}^i x \right) = \sum\limits_{i=1}^{n}E_i(t_n \boldsymbol{\zeta}_{2n,n}^+) x^{i}. 
\]
By assumption, $E_i(t_n \boldsymbol{\zeta}_{2n,n}^+)$ converges to $c_{i}$, and by~(\ref{e:ekzeta}) we have\[
  E_i(t_n \boldsymbol{\zeta}_{2n,n}^+) = t_{n}^{i}\frac{\prod_{r=n-i+1}^n\sin(\frac{r\pi}{2n})}{\prod_{j=1}^{i}(\sin(\frac{j\pi}{2n}))}.
\]
By again approximating the $\sin$ factors for large $n$, we have that $\lim_{n\to\infty}(\sin(\frac{r\pi}{2n}))=1$ and $\lim_{n\to\infty}(n\sin(\frac{j\pi}{2n}))=\frac{j\pi}{2}$. Hence 
\[
  \lim_{n\to\infty} E_i(t_n\, \boldsymbol{\zeta}_{2n,n}^+) = \lim_{n\to\infty}(nt_{n})^{i}\frac{\prod_{r=n-i+1}^n\sin(\frac{r\pi}{2n})}{\prod_{j=1}^{i}(n\sin(\frac{j\pi}{2n}))}=\lim\limits_{n\to\infty} \frac{(nt_n)^i2^i }{i! \pi^{i}},
\]
so $(nt_{n})_n$ must be convergent. Let $\lim\limits_{n\to \infty}nt_n = \gamma$. Then $c_{i} = \frac{\gamma^i 2^i}{i! \pi^{i}}$, so \[
  p_{u^{\infty}}(x) = e^{\eta x},  
\]
 where $\eta = \frac{2\gamma}{x}$, as desired. 
\end{proof}

\begin{remark}\label{r:otherexp} In the proof of Proposition~\ref{p:exponential} we considered the sequence  $u^{(2n)}=u_{2n}(t_n\boldsymbol{\zeta}^+_{2n,n})$ given by $t_n=\frac \gamma n$, and computed the limit $n\to\infty$ explicitly, obtaining an exponential function. However,  there is an alternative approach to arguing that this limit must result in an exponential function. Consider the generating function~\eqref{e:n2ngen}, and note that, for pairs of factors that are complex conjugate, we have the coefficient-wise inequality, 
\[(1+t_n\zeta^j_{2n}x)(1+t_n\zeta^{n-j+1}_{2n}x)=1+(\zeta^j_{2n}+\zeta^{n-j+1}_{2n})t_n x+t_n^2x^2<\left(1+t_nx\right)^2.
\]
Since $t_n=\frac{\gamma}n$ for some $\gamma\in\R_{\ge 0}$, the above inequality implies
\begin{equation}\label{e:preexpbound}
p_{u^{(2n)}}(x)=\prod_{j=1}^n(1+\frac{\gamma}n\zeta^j_{2n}x)<\left(1+\frac{\gamma x}n\right)^n.
\end{equation}
Therefore, possibly passing to a subsequence, we find an infinite Toeplitz matrix $u^\infty$ in the limit,  
and its generating function $p_{u^\infty}$ must satisfy $p_{u^\infty}(x)\le e^{\gamma x}$.

Next, use that the set
\[
\boldsymbol{\zeta}^+_{2n,n}\cup 
-\boldsymbol{\zeta}^+_{2n,n}=\{\pm\zeta^i_{2n}\mid i=1,\dotsc,n\}
\]
is precisely the set of all solutions to $x^{2n}+(-1)^{n}=0$, from which it follows that 
\[
\prod_{i=1}^n(1+\zeta^i_{2n}x)(1-\zeta^i_{2n}x)=1 + (-1)^n x^{2n}.
\]
This implies for our polynomial $p_{u^{(2n)}}(x)$ that
\begin{equation}
p_{u^{(2n)}}(x)p_{u^{(2n)}}(-x)=\prod_{i=1}^n(1+t_n\zeta^i_{2n}x)(1-t_n\zeta^i_{2n} x)=1 + (-1)^n t_n^{2n} x^{2n}.
\end{equation}
Therefore taking the limit $n\to\infty$ we see that $p_{u^\infty}(x)$ satisfies the symmetry property,
\begin{equation}\label{e:inflimidentity}
p_{u^\infty}(x)=\frac{1}{p_{u^\infty}(-x)}.
\end{equation}
We note, however, that $p_{u^\infty}(x)$ has no poles, since it is bounded by $e^{\gamma x}$ as a consequence of \eqref{e:preexpbound}. On the other hand, the equality \eqref{e:inflimidentity} then implies that $p_{u^\infty}(-x)$ has no zeroes. Therefore, $p_{u^\infty}(x)$ has no zeros. Therefore the Schoenberg parameters for $u^\infty$ are all $0$ and Edrei's theorem implies that $p_{u^\infty}(x)=e^{\eta x}$ for some $\eta\ge 0$.  
Finally, we may again obtain $\eta$ by computing the coefficient of $x$ in $p_{u^{(2n)}}(x)$ and taking the limit $n\to \infty$. Namely, with $v=e^{\frac{\pi i}{n}}$ we have,
\begin{equation}\label{e:eta}
\eta=\lim_{n\to\infty}\frac{\gamma}n
E_1(\boldsymbol{\zeta}^+_{2n,n})=\lim_{n\to\infty}\frac{\gamma}n[n]_v={\gamma}\lim_{n\to\infty}\frac{\sin(\frac{\pi}{2})}{n\sin(\frac{\pi}{2n})}=2\gamma/\pi.    
\end{equation}
\end{remark}
The key part of Edrei's theorem (that the non-vanishing, entire factor in \eqref{eq:inf-gen-func} must be of the form $e^{\gamma x}$) can now be equivalently restated using ``$Gr(n,2n)$-Toeplitz matrices", thanks to Proposition~\ref{p:exponential}, giving the following result.
 \begin{corollary}\label{c:Edreiequiv}
Any element of $u^\infty\in\Toep_\infty(\R_{\ge 0})$ whose generating function has no zeros or poles must lie in the limiting set of the Peterson Toeplitz varieties of middle-dimensional Grassmannians. That is, $u^\infty=\lim_{n\to\infty}u^{(2n)}$ for a sequence  $(u^{(2n)})_n$ with  $u^{(2n)}\in X^{(2n)}_{P_n}(\R_{>0})$.
\end{corollary}
\begin{proof}
By Theorem~\ref{theorem:edrei-infinite}, if the generating function of $u^\infty$ has no zeros or poles then it must be of the form $e^{\gamma x}$, and then Proposition~\ref{p:exponential} says it is the limit of a sequence of elements $u^{(2n)}\in X_{P_n}^{(2n)}(\R_{>0})$. 
\end{proof}
\begin{remark}
A direct proof of Corollary~\ref{c:Edreiequiv} would give rise to a new proof of Edrei's theorem, since, again by Proposition~\ref{p:exponential},   if $u^\infty$ is such a limit of $u^{(2n)}\in X^{(2n)}_{P_n}(\R_{>0})$ then it must have generating function of the form $e^{\gamma x}$.  

\end{remark}

\section{Strange duality} \label{s:strange} We observe a new perspective on the so-called `strange duality' involution for Grassmannian quantum cohomology, as arising within  Peterson theory, and use it to give alternative proofs of Theorems~\ref{thm:schub-class-lim} and \ref{thm:schub-class-lim-k}. Recall first that strange duality is an involution of $qH^*(Gr(k,n))[q_k\inv]$ that inverts the quantum parameter. It was found by Postnikov~\cite{PostnikovAffine}, and independently in the case of $q=1$ by Hengelbrock~\cite{hengelbrock}, who interpreted it in terms of complex conjugation. Let $\operatorname{MaxDiag}(\lambda)$ denote the number of boxes contained in the diagonal of the Young diagram~$\lambda$, and recall that the Poincar\'e dual $PD(\lambda)$ of $\lambda\subseteq k\times{m}$ is the complement of $\lambda$ inside the $k\times m$ box, rotated by $180$ degrees.
\begin{definition}[{\cite[Definition~2.2]{hengelbrock},\cite{PostnikovAffine}}]
For a partition $\lambda\subseteq k\times m$, let $\mathbb{S}(\lambda)$ be the partition obtained by considering the maximal square in $\lambda$, whose diagonal has $\mu:=\operatorname{MaxDiag}(\lambda)$ many boxes, and replacing the two parts of $\lambda$ that lie outside this square by their Poincar\'e duals (relative to their natural bounding rectangles of dimension $\mu\times (m-\mu)$ and $(k-\mu)\times \mu$, respectively).   
\end{definition}
\begin{remark} Note that the partition $\mathbb S(\lambda)$ still has the same $\operatorname{MaxDiag}$ value, and $\mathbb S$ is an involution on partitions $\lambda\subseteq k\times m$ that preserves $\MaxDiag$.
\end{remark}
\begin{theorem}[{\cite{PostnikovAffine,hengelbrock}}] The linear map defined by 
\begin{eqnarray*}
\mathbf S: qH^*(Gr(k,n))[q\inv]&\longrightarrow & qH^*(Gr(k,n))[q\inv]\\
q^\ell\sigma^\lambda_{P_k}&\mapsto&
q^{-\ell-\operatorname{MaxDiag}(\lambda)}\sigma_{P_k}^{\mathbb{S}(\lambda)}.
\end{eqnarray*}  
 is an algebra isomorphism. 
\end{theorem}
\begin{remark} It is straightforward that $\mathbf S$ is an involution, and $\mathbf S$ is referred to as `strange duality'.  We remark that there is also an interesting generalisation to other types  \cite{ChaputManivelPerrinII,ChaputManivelPerrinIII}, and  a construction via the homology of the affine Grassmannian and (another) Peterson isomorphism \cite{CookmeyerMilicevic}. \end{remark}

\begin{definition}\label{d:strangeYP} We define an automorphism of the Peterson variety $Y^\circ_{P_k}$ by
\begin{eqnarray}\label{e:strangeYP}
    \tau:\qquad  Y^\circ_{P_k}\ &\overset\sim\longrightarrow & Y^\circ_{P_k}\\
    yw_0 B_-&\mapsto& y^T w_{P_k} B_-.\nonumber
\end{eqnarray}
Here, the fact that $y^T w_{P_k} B_-$ is an element of $Y^\circ_{P_k}$ follows from~\cite[Section~8]{rietsch2001flagvarieties}. Note that this is specific to the Grassmannian setting and does not generalise in this form to other partial flag varieties.
\end{definition}

\begin{proposition}\label{p:strangeduality}
Let us identify $\C[Y^\circ_{P_k}]$ with $qH^*(Gr(k,n),\C)[q\inv]$ via Peterson's isomorphism from Theorem~\ref{t:PetersonY}. Then $\tau^*$ recovers the strange duality map $\mathbf S$, giving a commutative diagram

\[\begin{tikzcd}
	{\C[Y_{P_k}^\circ]} && {\C[Y_{P_k}^\circ]} \\
	{qH^*(Gr(k,n))[q\inv]} && {qH^*(Gr(k,n))[q\inv].}
	\arrow["{\tau^*}", from=1-1, to=1-3]
	\arrow["{\cong}", from=2-1, to=1-1]
	\arrow["{\mathbf S}", from=2-1, to=2-3]
	\arrow["{\cong}"'from=2-3, to=1-3]
\end{tikzcd}\]
\end{proposition}
\begin{remark} Note the asymmetry in Definition~\ref{d:strangeYP}, with the left-hand side of \eqref{e:strangeYP} expressing an element of $Y^\circ_{P_k}$ as element of the big cell $B_-w_0B_-/B_-$ and the right-hand side as element of the Bruhat cell $B_+w_{P_k} B_-/B_-$. It is surprising that $\tau$ is an involution.  
\end{remark}
\begin{proof} We first claim that $\tau^*$ can be described by
\begin{equation}\label{e:taustar}
    \tau^*: \ G^\lambda_{P_k}\quad \mapsto\quad \frac{G^{PD(\lambda)}_{P_k}}{G^{{k\times m}}_{P_k}},
\end{equation} 
in terms of the functions $G^\lambda_{P_k}$ on $Y_{P_k}^\circ$ from Definition~\ref{d:Gwlambda}. 
To check this, note first that 
\begin{eqnarray}\label{e:tauGl}
G^\lambda_{P_k}(\tau(y w_0 B_-))&=&G^\lambda_{P_k}(y^T w_{P_k} B_-)= \langle y^T\cdot v_{k+1}\wedge\dotsc\wedge v_n,\dot w_\lambda \cdot v_{k+1}\wedge\dotsc\wedge v_n \rangle.
\end{eqnarray}
On the other hand we have 
\begin{eqnarray*}
G^{PD(\lambda)}_{P_k}(y \dot w_0 B_-)&=&\frac{\langle y\dot w_0\cdot v_{k+1}\wedge\dotsc\wedge v_n,\dot w_{PD(\lambda)} \cdot v_{k+1}\wedge\dotsc\wedge v_n \rangle}{\langle y\dot w_0\cdot v_{k+1}\wedge\dotsc\wedge v_n, v_{k+1}\wedge\dotsc\wedge v_n \rangle}
\\
&=&\frac{ \langle \dot w_0\inv y\dot w_0\cdot v_{k+1}\wedge\dotsc\wedge v_n,\dot w_0\inv \dot w_{PD(\lambda)} \cdot v_{k+1}\wedge\dotsc\wedge v_n \rangle}{\langle \dot w_0\inv y\dot w_0\cdot v_{k+1}\wedge\dotsc\wedge v_n, \dot w_0\inv\cdot v_{k+1}\wedge\dotsc\wedge v_n \rangle},
\end{eqnarray*}
and we have
\begin{eqnarray*}
G^{k\times m}_{P_k}(y \dot w_0 B_-)&=&\frac{1}{\langle \dot w_0\inv y\dot w_0\cdot v_{k+1}\wedge\dotsc\wedge v_n, \dot w_0\inv\cdot v_{k+1}\wedge\dotsc\wedge v_n \rangle}.
\end{eqnarray*}
As a consequence, 
\begin{eqnarray}\label{e:GPDoverGpt}
    \frac{G^{PD(\lambda)}_{P_k}(y \dot w_0 B_-)}{G^{k\times m}_{P_k}(y \dot w_0 B_-)}=\langle \dot w_0\inv y\dot w_0\cdot v_{k+1}\wedge\dotsc\wedge v_n,\dot w_0\inv \dot w_{PD(\lambda)} \cdot v_{k+1}\wedge\dotsc\wedge v_n \rangle.
\end{eqnarray}
Now consider $d_\epsilon$ as in the proof of Proposition~\ref{p:comparison}, for example, $d_\epsilon=\operatorname{diag(\pm 1,\dotsc, -1,1)}$. Since $y$ is Toeplitz, we have that $\dot w_0\inv y\dot w_0=d_\epsilon y^T d_\epsilon\inv$. We also have $w_{PD(\lambda)}=w_0w_\lambda \mod W_{P_k}$. Therefore, comparing \eqref{e:tauGl} and \eqref{e:GPDoverGpt}, we see that these two expressions differ at most by a sign, 
\begin{equation}
G^\lambda_{P_k}(\tau(y w_0 B_-))=\pm \frac{G^{PD(\lambda)}_{P_k}(y w_0 B_-)}{G^{k\times m}_{P_k}(y w_0 B_-)}.
\end{equation}
However, $\tau$ preserves total positivity, and the $G^\lambda_w$ are all positive on the totally positive part of $Y^\circ_{P_k}$. Therefore the sign must be positive, and the claim \eqref{e:taustar} is verified. 

Note that since $\tau^*$ is a well-defined algebra homomorphism, by transporting this formula to quantum cohomology we see that the point class $\sigma^{(k\times m)}_{P_k}$ must be an invertible element in $qH^*(Gr(k,n))[q\inv]$. This is already known.  It also follows also from a formula in quantum Schubert calculus given in \cite[Theorem~3.3]{ChaputManivelPerrinII}, which says
\begin{equation}\label{e:CMPformula}
\sigma_{P_k}^{k\times m}\star\sigma^{\lambda}_{P_k}=q^{\MaxDiag(\lambda)}\sigma_{P_k}^{PD({\mathbb S}(\lambda))}.
\end{equation}
Moreover, replacing $\lambda$ by ${\mathbb S}(\lambda)$ and reorganising \eqref{e:CMPformula} gives the following alternative formula for $\mathbf S$, 
see ~\cite{ChaputManivelPerrinIII}, which is
\begin{equation}
\mathbf S(\sigma_{P_k}^\lambda)=q^{-\MaxDiag(\lambda)}\sigma_{P_k}^{\mathbb S(\lambda)}=\frac{\sigma_{P_k}^{PD(\lambda)}}{\sigma_{P_k}^{k\times m}}.
\end{equation} This now implies the proposition.
  \end{proof}
  \begin{remark} \label{r:Slambdaeasy} We note that for $y=u^E_n(z_1,\dotsc, z_k)$ the equation \eqref{e:tauGl} leads to the explicit formula   $G^\lambda_{P_k}(\tau(y w_0 B_-))=G^\lambda_{P_k}(y^Tw_{P_k}B_-))=S_\lambda(z_1,\dotsc, z_k)$.
  \end{remark}
\begin{corollary}\label{c:strange} Suppose $\{z_1,\dotsc, z_k\}\in\mathcal X_{n,k}$. We have the  following identity,
\begin{equation}\label{e:StrangePeterson}
u^E_n(z_1\inv,\dotsc, z_k\inv)w_0B_-=u^E_n(z_1,\dotsc, z_k)^Tw_{P_k}B_-,
\end{equation}
that shows how to express an element of $Y^\circ_{P_{k}}$ in two equivalent ways.
\end{corollary}
\begin{proof}
 Recall the result of \cite{hengelbrock} that the strange duality automorphism $\sigma^{\lambda}\mapsto q^{-\MaxDiag(\lambda)} \sigma_{P_k}^{\mathbb{S}(\lambda)}$ represents complex conjugation when $q=1$, see also \cite[Section~3]{ChaputManivelPerrinIII}. For general $q$ it is straightforward to observe that this description extends, but with complex conjugation replaced by inversion on the level of the Chern roots $t_i$. As a consequence, the image $\mathbf S(\sigma^\lambda_{P_k})$ of the class $\sigma^\lambda_{P_k}=S_\lambda(t_1,\dotsc, t_k)$ in $qH^*(Gr(k,n),\C)[q\inv]$ has three different expressions, 
\begin{equation}\label{e:StrangeChern}
\mathbf{S}(\sigma^\lambda_{P_k})=\frac{S_{PD(\lambda)}}{S_{k\times m}}(t_1,\dotsc,t_k)=(-1)^{k+1}t_1^{-n\MaxDiag(\lambda)}S_{\mathbb{S}(\lambda)}(t_1,\dotsc,t_k)=S_\lambda(t_1\inv,\dotsc,t_k\inv),
\end{equation}
keeping in mind that $q$ in terms of the Chern roots $t_i$ is given by $(-1)^{k+1}t_1^n$, see Lemma~\ref{l:Ykm}.

We now apply the dictionary between $\C[Y_{P_k}^\circ]$ and $qH^*(Gr(k,n))[q\inv]$, whereby $\sigma^{\lambda}_{P_k}$ corresponds to $G^\lambda_{P_k}$ and the coordinates $z_i$ for  $u^E_n(z_1,\dotsc,z_k)w_{P_k}B_-\in Y_P^\circ$ correspond to the Chern roots $t_i$, compare Remark~\ref{r:Slambdaeasy}.

We now have the following two identities
\begin{eqnarray*}
\tau(u^E_n(z_1\inv,\dotsc,z_k\inv)w_0B_-)&=&u^E_n(z_1\inv,\dotsc,z_k\inv)^T w_{P_k}B_-,\\
\tau(u^E_n(z_1,\dotsc,z_k )^T w_{P_k}B_-)&=&u^E_n(z_1\inv,\dotsc,z_k\inv)^Tw_{P_k}B_-,
\end{eqnarray*}
where the first is just the definition of $\tau$, and the second follows from \eqref{e:StrangeChern} via Proposition~\ref{p:strangeduality}, see also Remark~\ref{r:Slambdaeasy}. Namely, the strange involution acts by inverting the Chern roots which correspond to the $z_i$ under our dictionary. Comparing the left-hand sides and using the fact that $\tau$ is invertible, we obtain the desired formula~\eqref{e:StrangePeterson}.
\end{proof}

We now apply these results to give an alternative proof of our results  concerning the Schubert class asymptotics.
\begin{proof}[Proof of Theorems~\ref{thm:schub-class-lim} and \ref{thm:schub-class-lim-k}]
 Consider the isomorphism $\mu_{P_k}:X_{P_k}\to Y^\circ_{P_k}$ from Proposition~\ref{p:comparison} given by $\mu_{P_k}(u)=\iota(u)w_0B_-$, and suppose $u=u^E_n(x_1,\dotsc, x_m)$ where $\{x_1,\dotsc, x_m\}\in \mathcal X_{n,m}$, as in Lemma~\ref{l:uE}.  
Extending $\{x_1,\dotsc, x_m\}$ to an element $[x_1,\dotsc,x_m,x_{m+1},\dotsc, x_{n}]_{m,k}\in \mathcal Y_{m,k}$ and combining Lemma~\ref{l:symm}, Remark~\ref{r:qsymm} and Corollary~\ref{c:strange}, we find that 
\[
\mu_{P_k}(u)=\iota(u^E_n(x_1,\dotsc,x_m))w_0B_-=u^E_n(-x_{m+1},\dotsc,-x_n)w_0B_-=u^E_n(-x_{m+1}\inv,\dotsc,-x_n\inv)^Tw_{P_k}B_-.
\]
Moreover, using also Proposition~\ref{p:comparison}, Remark~\ref{r:Slambdaeasy} and Lemma~\ref{l:Ykm}, this expression for $\mu_{P_k}(u)$ leads to the following explicit formula for the evaluation of the functions $\mathfrak S^{\lambda}_{P_k}$ on $u=u^E(x_1,\dotsc,x_m)\in X_{P_k}$,
\begin{equation}\label{e:strangeeval}
\mathfrak S^\lambda_{P_k}(u^E(x_1,\dotsc, x_k))=G^\lambda_{P_k}(\mu_{P_k}(u^E(x_1,\dotsc, x_k)))=S_\lambda(-x_{m+1}\inv,\dotsc,-x_n\inv)=S_{\lambda'}(x_1\inv,\dotsc,x_m\inv).
\end{equation}
Now suppose $u^{(n)}=u^{+}_n(t_n)\in X_{P_{n-m}}(\R_{>0})$ converges to $u^{\infty}$ as in Proposition~\ref{prop:lim-q-param}. We then have $\beta=\lim_{n\to\infty}(t_n)$ and $p_{u^\infty}(x)=(1+\beta x)^m$. Also assume $\beta\ne 0$. Since $u^+_n(t_n)=u^E_n(x_1,\dotsc, x_m)$ for $x_j=t_n\zeta^j_{n,m}$, we have that \eqref{e:strangeeval} implies
\[
\mathfrak S^\lambda_{P_{n-m}}(u^{(n)})=S_{\lambda'}((t_n\zeta^1_{n,m})\inv,\dotsc,(t_n\zeta^m_{n,m})\inv).
\]
Now, since $t_n\to\beta$ and $\zeta^j_{n,m}\to 1$  as $n\to\infty$ and $m$ is fixed, it follows that $\lim_{n\to\infty}\left(\mathfrak S^\lambda_{P_{n-m}}(u^{(n)})\right)= S_{\lambda'}(\frac 1\beta,\dotsc, \frac 1\beta)$, where $S_{\lambda'}$ is a Schur polynomial in $m$ variables.  

Finally, consider the case where $k$ is fixed and $u^{(n)}=u^E_n(t_n\boldsymbol{\zeta}^+_{n,n-k})$. If $u^{(n)}$ converges, then its limit is $u^\infty$ with $p_{u^\infty}(x)=\frac{1}{(1-\alpha x)^k}$ for  $\alpha=\lim_{n\to\infty}(t_n)$. Let us again assume $\alpha\ne 0$. Now we consider the alternative expression for $\mathfrak S^\lambda_{P_k}(u^{(n)})$ in terms of $S_\lambda$, again from \eqref{e:strangeeval}, which is 
\[
\mathfrak S^\lambda_{P_{k}}(u^{(n)})=S_{\lambda}((t_n\zeta^1_{n,k})\inv,\dotsc,(t_n\zeta^k_{n,k})\inv),
\]
using that if $(x_1,\dotsc,x_m)=t_n\boldsymbol{\zeta}^+_{n,m}$, then $(-x_{m+1},\dotsc,-x_n)=t_n\boldsymbol{\zeta}^+_{n,k}$, compare Lemma~\ref{l:symm}. Now $k$ is fixed and we can straightforwardly take the limit, which gives the desired fomula $\lim_{n\to\infty}\left(\mathfrak S^\lambda_{P_k}(u^{(n)})\right)= S_{\lambda}(\frac 1\alpha,\dotsc, \frac 1\alpha)$. 
\end{proof}
\begin{remark} 
Note that Remark~\ref{r:example} contains an example of the formula \eqref{e:strangeeval}. We may also use \eqref{e:strangeeval} to finally interpret the $x_i$ from Lemma~\ref{l:uE}. Namely,  in terms of the isomorphism with quantum cohomology from Theorem~\ref{t:Peterson}, recall that $\mathfrak S^\lambda_{P_k}$ represents $\sigma^{\lambda}_{P_k}=S_\lambda(t_1,\dotsc t_k)=S_{\lambda'}(-t_{k+1},\dotsc,-t_n)$, where $t_1,\dots, t_k$ are the Chern roots of the dual tautological subbundle, and $-t_{k+1},\dotsc, -t_n$ the Chern roots of the tautological quotient bundle. See also Section~\ref{s:qcoh}. We conclude that our `coordinates' $x_i$ for $u^E_n(x_1,\dotsc, x_m)\in X_{P_{m}}$ should therefore be considered as representing the inverted Chern roots of the tautological quotient bundle of $Gr(k,n)$.
\end{remark}
\section{Partial flag varieties and conjectures}\label{sec:conjectures}

 For a finite strongly increasing sequence of positive integers $\mathbf n=(n_1,\dotsc n_k)$,  let  $P_{\mathbf n}^{(n)}\coloneq P_{n_1,\dotsc,n_k}^{(n)}$ be the parabolic subgroup associated to the partial flag variety $SL_n/P_{\mathbf n}^{(n)}=\Fl_{n_1,\dotsc,n_k}(\C^n)$ (for $n>n_k$). We begin by stating the following conjectural analogue of Theorem~\ref{thm:schub-class-lim} for these more general parabolic subgroups.

\begin{conjecture}\label{con:Fin}
Fix $\mathbf n$ as above, and set $n_0=0$.
Consider a monotone (weakly) decreasing sequence of $n_k$ non-negative real numbers  $\left( \alpha_i \right)_{i=1}^{n_k}$ such that  $\alpha_{n_i}> \alpha_{n_i+1}$,  but $\alpha_j=\alpha_{j+1}$ otherwise. 
Let $u^{\infty}\in \mathrm{Toep}_{\infty}(\R_{\ge 0})$ be the infinite totally nonnegative Toeplitz matrix with generating function  
\begin{equation}\label{e:con1}
    p_{u^\infty }(x)=  \frac{1}{\prod\limits_{j=1}^{n_k}(1-\alpha_jx)} =\frac{1}{\prod\limits_{i=1}^{k}(1-\alpha_{n_i}x)^{n_i -n_{i-1}}}.
\end{equation}
Then there exists a sequence of matrices $u^{(n)}$, with $u^{(n)}\in X_{P_{\mathbf n}^{(n)}}(\R_{> 0})$, which converges to $u^{\infty}\in \mathrm{Toep}_{\infty}(\R_{\ge 0})$.
Conversely, every Toeplitz matrix $u^\infty$ with $p_{u^\infty}(x)$ as in \eqref{e:con1} arises as the limit of such a sequence.  

\end{conjecture}

\begin{remark}\label{r:con1dual}
We could have also stated a dual version of this conjecture, instead considering sequences $u^{(n)}$ where $u^{(n)}\in X_{P^{(n)}_{n-\mathbf m}}$, each corresponding to the partial flag variety with fixed codimensions, but this, as well as Conjecture~\ref{con:Fin}, is another special case of Conjecture~\ref{c:pflagconjectures-alt}.
\end{remark}
\begin{definition}\label{d:infiniteflagseqs}
Consider the set $\mathcal S$ of pairs $(\mathbf n,\mathbf m)$ of strongly increasing positive integer sequences, which may be finite or infinite. We write $\mathbf n=(n_1,n_2, n_3,\dotsc )$ and $\mathbf m=(m_1,m_2, m_3,\dotsc )$ for the sequences $\mathbf n$ and $\mathbf m$, respectively. We have four cases, according to which of the sequences are infinite. 

\begin{enumerate}
\item Suppose $(\mathbf n,\mathbf m)$ consists of two finite sequences (of length $k,\ell\ge 0$, respectively). Then we set
\[
\Fl_{(\mathbf n,\mathbf m)}(\C^n):=
\Fl_{n_1,\dotsc, n_k,n-m_\ell,\dotsc, n-m_1}(\C^n)
\]
whenever $n$ is large enough. 
\item If only $\mathbf n$ is finite, we set
\[
\Fl_{(\mathbf n,\mathbf m)}(\C^n):=
\Fl_{n_1,\dotsc, n_k,n-m_r,\dotsc, n-m_1}(\C^n)
\]
where $r$ is maximal for $n-m_r>n_k$ (and $n$ large enough for $r$ to exist). 
\item If only  $\mathbf m$ is finite, we set
\[
\Fl_{(\mathbf n,\mathbf m)}(\C^n):=
\Fl_{n_1,\dotsc, n_r,n-m_\ell,\dotsc, n-m_1}(\C^n)
\]
where $r$ is maximal for $n_r<n-m_\ell$ (and $n$ large enough for $r$ to exist). 
\item Given a pair of infinite sequences $(\mathbf n,\mathbf m)\in \mathcal S$ and
$n\in\N$, let us write 
\[
\Fl_{(\mathbf n,\mathbf m)}(\C^n):=
\Fl_{n_1,\dotsc, n_r,n-m_r,\dotsc n-m_1}(\C^n)
\]
where $r$ is the maximal index for which  $n_r<n-m_r$. 
\end{enumerate}
Let $P^{(n)}_{(\mathbf n,\mathbf m)}$ denote the associated parabolic in $SL_{n}$. 
\end{definition}

\begin{conjecture}\label{c:pflagconjectures-alt}
Suppose $(\mathbf n,\mathbf m)\in\mathcal S$ and recall the notations from above.
Consider a pair of infinite sequences  $\boldsymbol{\alpha}_{\mathbf{n}} \coloneq 
{\left(\alpha_i\right)}_{i\in \N}$ and $\boldsymbol{\beta}_{\mathbf{m}}\coloneq 
{\left( \beta_i \right)}_{i\in \N}$ in $\R_{\ge 0}$ satisfying the following conditions:
\begin{enumerate}
\item\label{item:1}  $\alpha_{n_i}> \alpha_{n_i+1}$ and $\beta_{m_i}> \beta_{m_i+1}$,
\item\label{item:2} $\alpha_j=\alpha_{j+1}$ (resp.\ $\beta_j = \beta_{j+1}$) for all $j\notin\{n_{1},n_2,\dotsc \}$ (resp.\ $j\notin  \{m_{1},m_2,\dotsc\}$),
 \item\label{item:4} $\sum \alpha_{i} <\infty$ and $\sum \beta_{i}<\infty$, 
  \item\label{item:3} If $\mathbf{n}$ (resp.\ $\mathbf{m}$) is finite of length $k$ (resp.\ $\ell$), then $\alpha_{j}=0$ for all $j>n_{k}$ and \ $\beta_{j}=0$ for all $j>m_{\ell}$.
  \end{enumerate}
  Define the infinite Toeplitz matrix $u^{\infty}\in \mathrm{Toep}_{\infty}(\R_{\ge 0})$ by its generating function \[
      p_{u^{\infty}}(x) = \frac{\prod\limits_{\beta\in \boldsymbol{\beta}_{\mathbf{m}}
        }(1 + \beta x)}{\prod\limits_{\alpha \in \boldsymbol{\alpha}_{\mathbf{n}}
        }(1-\alpha x)}.
    \]
Then there exists a sequence $(u^{(n)})_n$ with $u^{(n)}\in X_{P_{(\mathbf{n},\mathbf{m})}^{(n)}}(\R_{> 0})$ that converges to $u^{\infty}$. 
\end{conjecture}

\subsection{Full flag asymptotics review}

In Section~\ref{s:pflag-limits}, we will partially generalise Proposition~\ref{prop:lim-q-param} to a setting involving more general parabolic subgroups. The strategy of the proof will involve applying results from the full flag case obtained in ~\cite{rietsch2025totallypositivetoeplitzmatrices}. We recall the relevant results from ~\cite{rietsch2025totallypositivetoeplitzmatrices} below. Let $B_n$ denote the upper-triangular Borel subgroup of $SL_n$.

Recall first Remark~\ref{r:diforfullflag}, where we gave a version of Theorem~\ref{t:totposGP} that takes the form 
\begin{eqnarray}\label{e:dinparam}
  X_{B_{n}}(\R_{>0})&\xrightarrow{\sim}& T_{\mathrm{SL}_n}(\R_{>0})
\\
u\quad \ &\mapsto & \operatorname{diag}(d_1^{(n)}, d_2^{(n)},\dotsc, d_n^{(n)}),\notag
\end{eqnarray}
with $d_i^{(n)}=\frac{\Delta_{i}(u)}{\Delta_{i-1}(u)} $. For ease of notation we write $\Delta_{i}$ (resp.\ $d^{(n)}_{i}$) when it is clear we are taking the evaluation at a given point $u$, and we will leave out the superscript $(n)$ when $n$ is clear from context. Further composing \eqref{e:dinparam}  with the homeomorphism sending $\mathrm{diag}(d_{1},\dots,d_{n})\in T_{SL_n}(\R_{>0})$ to $ \left( \frac{d_{2}}{d_1},\dots,\frac{d_n}{d_{n-1}} \right)\in \R^{n}_{>0}$ gives a parameterisation \[
  (q_1, \dots, q_{n-1}): X_{B_{n}}(\R_{>0})\longrightarrow \R^{n}_{>0},
\]
 where $q_{i}(u) = \frac{d_{i+1}(u)}{d_{i}(u)}$, see~\cite[Theorem 2.1]{rietsch2025totallypositivetoeplitzmatrices}, where the theorem is stated in similar conventions. Note that this  recovers  Kostant's expression for the quantum parameters,  
 \[
   q_{i} = \frac{\Delta_{i+1}\Delta_{i-1}}{\Delta_{i}^{2}}, 
 \]
which is an instance of the formula~\eqref{e:qviadelta} we recalled earlier.

We can now state the following result from~\cite{rietsch2025totallypositivetoeplitzmatrices}, which relates closely to Theorem~\ref{t:fullflag}.
\begin{theorem}[{\cite[Theorem~2.7]{rietsch2025totallypositivetoeplitzmatrices}}]
\label{t:fflag-d-lim}Suppose $(u^{(n)})_n$ is a sequence of (finite) totally positive matrices $u^{(n)}\in X_{B_n}(\R_{>0})$ that converges uniformly to an element $u^{\infty}\in \mathrm{Toep}_{\infty}(\R_{\ge 0})$, and assume that the Schoenberg parameters of $u^{\infty}$ are all nonzero, so that 
\[
p_\infty(x)=e^{\gamma x}\frac{\prod_{j=1}^\infty(1+\beta_i x)}{\prod_{i=1}^\infty(1-\alpha_j x)}
\]
has infinitely many roots and poles. Let $(d_{1}^{(n)},\dots, d_{n}^{(n)})$ be the positive parameters associated to $u^{(n)}$ as in \eqref{e:dinparam}. Then 
\begin{equation}
\label{e:fflag-d-lim}
\lim\limits_{n\to \infty} \sqrt[n]{d_{i}^{(n)}(u^{(n)})} = \alpha_{i} \qquad \text{and}\qquad \lim\limits_{n\to \infty} \sqrt[n]{d_{n-j+1}^{(n)}(u^{(n)})} = \frac{1}{\beta_{j}}.
\end{equation} 
\end{theorem}

\begin{remark}
Note that Theorem~\ref{t:fullflag} is a corollary of Theorem~\ref{t:fflag-d-lim} above, with the formulas  concerning the limit of $\sqrt[n]{q_{i}^{(n)}}$ and $\sqrt[n]{q_{n-i}^{(n)}}$ following via the identification $q_{i} = \frac{d_{i+1}}{d_{i}}$. 
\end{remark}
\begin{remark}
In the case of the Grassmannian $Gr(k,n)$ we also have a diagonal matrix $d\in T_{SL_n}^{W_{P_k}}$ governing the quantum parameters, see Section~\ref{s:qparam} and \eqref{e:qni}. We may therefore restate the results from Propositions \ref{prop:lim-q-param} and \ref{prop:lim-q-param2} concerning asymptotics of the quantum parameter as results about asymptotics of the entries of $d$ and observe that the resulting formulas are in fact directly analogous to the ones in Theorem~\ref{t:fflag-d-lim}. Namely, suppose $u^{(n)}\in X_{P_k}(\R_{>0})$, and let us write 
\[ d=\operatorname {diag}(\underbrace{d^{(n)}_1,\dotsc, d^{(n)}_1}_{k},\underbrace{d^{(n)}_2,\dotsc, d^{(n)}_2}_{n-k})
\]
for the matrix of positive parameters associated to $u^{(n)}$. Then, since $q^{(n)}=\frac {d^{(n)}_2}{d^{(n)}_1}$ and $\det(d)=1$, we have 
\[
\sqrt[n]{q^{(n)}}=\sqrt[n-k]{\frac{1}{d_1^{(n)}}}=\sqrt[n-m]{d_2^{(n)}},
\]
where $m:=n-k$. Now assume a sequence of $u^{(n)}\in X_{P_{n-m}}(\R_{>0})$ converges to a matrix $u^\infty$ with $\beta_1=\dotsc=\beta_m=\beta>0$ and all other Schoenberg parameters equal to $0$, as in Proposition~\ref{prop:lim-q-param}. Then the result that $\sqrt[n]{q^{(n)}_{n-m}(u^{(n)})\inv}$ converges to $\beta$ can be reinterpreted as 
\begin{equation}\label{e:limd2}
\lim_{n\to\infty}\sqrt[n]{d_2^{(n)}(u^{(n)})}=\frac{1}\beta.
\end{equation}
Note that if we write $d^{(n)}_n:=d^{(n)}_2$, considering this as the final matrix entry of $d$, the limit~\eqref{e:limd2} is completely analogous to the second formula in~\eqref{e:fflag-d-lim}. Similarly, if $u^{(n)}\in X_{P_{k}}(\R_{> 0})$ converges to a matrix $u^\infty$ with $\alpha_1=\dots=\alpha_k=\alpha>0$ and all other Schoenberg parameters are equal to $0$, as in Proposition~\ref{prop:lim-q-param2}, then the limit formula for $\alpha$ implies that
\begin{equation}\label{e:limd1}
\lim_{n\to\infty}\sqrt[n]{d_1^{(n)}(u^{(n)})}=\alpha,
\end{equation}
in terms of $d_1^{(n)}$. This is directly analogous with the first equality in \eqref{e:fflag-d-lim}.
\end{remark}

\subsection{Asymptotics of arbitrary partial flag varieties}\label{s:pflag-limits}
   Let $(\mathbf n,\mathbf m)$ be as in Definition~\ref{d:infiniteflagseqs}. We consider the subset $W^{P_\mathbf{n}}$ of $S_\infty=\bigcup_{n=1}^\infty S_n$ consisting of the $w\in S_\infty$ all of whose reduced expressions end in elements of the form $s_{n_i}$ with $n_i\in\mathbf n$. Similarly for $\mathbf m$. 
   Given $w\in S_\infty$ and $n$ large enough so that $w\in S_n$, we set $\hat{w}_n = w_0^{(n)} w {( w_0^{(n)})}^{-1}$, where $w_0^{(n)}$ is the longest element in $S_n$, see Theorem~\ref{t:fullflag}. Observe that if $w\in W^{P_{\mathbf m}}$ then  $\hat{w}_n \in W^{P_{n - \mathbf{m}}^{(n)}}$. 
   Also, recall that $S_w$ denotes the Schubert polynomial associated to $w$. If $N$ is minimal such that $w\in S_N$, then $S_w$ is a polynomial in $N-1$ variables.  Below we will write $\mathfrak S^w_{(n)}$  for the function $\mathfrak S^w_{P_{(\mathbf{n},\mathbf{m})}^{(n)}}$ on $X_{P_{(\mathbf{n},\mathbf{m})}^{(n)}}$ to simplify notation. 

\begin{conjecture}\label{c:pflag-schub-lim}

  Given $(\mathbf n,\mathbf m)$ as above,  consider a pair of infinite sequences $\boldsymbol{\alpha}_{\mathbf{n}}$ and $\boldsymbol{\beta}_{\mathbf{m}}$ in $\R_{>0}$ satisfying the conditions (1)-(3) from Conjecture~\ref{c:pflagconjectures-alt}. Assume we have a sequence $(u^{(n)})_n$ with $u^{(n)}\in X_{P_{(\mathbf{n},\mathbf{m})}^{(n)}}(\R_{> 0})$ converging uniformly to $u^{\infty}$, where
   \[
      p_{u^{\infty}}(x) = \frac{\prod\limits_{\beta\in \boldsymbol{\beta}_{\mathbf{m}}
        }(1 + \beta x)}{\prod\limits_{\alpha \in \boldsymbol{\alpha}_{\mathbf{n}}
        }(1-\alpha x)}=\frac{\prod\limits_{j=1}^\infty(1+\beta_i x)}{\prod\limits_{i=1}^\infty(1-\alpha_j x)},
    \]
with the Schoenberg parameters given by $\boldsymbol{\alpha}_{\mathbf{n}}$ and $\boldsymbol{\beta}_{\mathbf{m}}$. Given a permutation $w\in W^{P_{\mathbf{m}}}$, 
  we have the following limit formula,
  \begin{equation} \label{eq:schub-lim-pflag}
    \lim\limits_{n\to \infty}{\left( \mathfrak{S}^{\hat w_n}_{(n)}(u^{(n)}) \right)}_n = S_w\left( \frac{1}{\beta_1},\frac{1}{\beta_2},\frac{1}{\beta_3}, \dots \right).        \end{equation}
 If $w\in W^{P_{\mathbf{n}}}$, then we have the limit formula 
 \begin{equation}\label{eq:schub-lim-pflag2}
     \lim\limits_{n\to \infty}{\left( \mathfrak{S}^{w}_{(n)}(u^{(n)}) \right)}_n = S_w\left( \frac{1}{\alpha_1},\frac{1}{\alpha_2}, \frac{1}{\alpha_3},\dots \right) .
 \end{equation}
    
\end{conjecture}

\begin{proposition}\label{p:pflag-q-lim-inf}
Let $(\mathbf n,\mathbf m)$ be as in Definition~\ref{d:infiniteflagseqs}, and assume both sequences are infinite. Consider a pair of sequences $\boldsymbol{\alpha}_{\mathbf{n}}$ and $\boldsymbol{\beta}_{\mathbf{m}}$ as in Conjecture~\ref{c:pflagconjectures-alt}, and suppose we have a sequence $u^{(n)}\in X_{P_{(\mathbf{n},\mathbf{m})}^{(n)}}(\R_{> 0})$ converging  to $u^{\infty}$ with Schoenberg parameters given by $\boldsymbol{\alpha}_{\mathbf{n}}$ and $\boldsymbol{\beta}_{\mathbf{m}}$. Assume, moreover, that $|\Delta_{n_i}(u^{(n)})-\Delta_{n_i}([u^\infty]_n) |$ and $|\Delta_{n-m_j}(u^{(n)})-\Delta_{n-m_j}([u^\infty]_n) |$ converge to $0$ as $n\to\infty$. Then we have   
\begin{equation}\label{eq:pflag-q-lim}
  \lim\limits_{n\to \infty} \sqrt[n]{q_{n_{i}}^{(n)}(u^{(n)})} = \frac{\alpha_{n_{i+1}}}{\alpha_{n_{i}}},\quad\text{ and }\quad \lim\limits_{n\to \infty} \sqrt[n]{q_{n-m_{i}}^{(n)}(u^{(n)})} = \frac{\beta_{m_{i+1}}}{\beta_{m_{i}}}.
\end{equation}
\end{proposition}

\begin{proof}
Let $P^{(n)}=P_{(\mathbf{n},\mathbf{m})}^{(n)}$. We have that $u^{(n)}\in X_{P^{(n)}}(\R_{>0})$ converges to $u^{\infty}$. 
Consider the truncations ${\left[ u^{\infty} \right]}_n$, and note that these lie in $X_{B_n}(\R_{>0})$. Recall that  both $X_{P^{(n)}}$ and $X_{B_n}$ are contained in $X_{\{n_{i}\}}$, see Definition~\ref{d:XJ}, whenever $n$ is large enough. For such $n$, we consider the function $\mathbf q_{n_i}^{(n)}$ on $X_{\{n_{i}\}}(\R_{>0})$ given by 
\begin{equation}\label{e:qnin}
\mathbf q_{n_i}^{(n)}:= \sqrt[{(n_{i+1}-n_i)(n_i- n_{i-1})}]{\frac{\Delta_{n_{i-1}}^{n_{i+1}-n_i}\Delta_{n_{i+1}}^{n_i-n_{i-1}}}{\Delta_{n_i}^{n_{i+1}-n_{i-1}}}}.
\end{equation} 
Note that $\mathbf q_{n_i}^{(n)}|_{X^{(n)}_{P^{(n)}}}= q_{n_i}^{(n)}$ by \eqref{e:qviadelta}. The difference between $\mathbf q_{n_i}^{(n)}$ and $q_{n_i}^{(n)}$ is the extended domain. Recall that for ${\left[ u^{\infty} \right]}_n \in X_{B_{n}}(\R_{>0})$, we have \[
  \Delta_\ell ({\left[ u^{\infty} \right]}_n)= \prod\limits_{j=1}^{\ell}d^{(n)}_{j},
\]
where $(d_1^{(n)}, \dots, d_n^{(n)})$ are the positive parameters associated to $[u^\infty]_n$ in \eqref{e:dinparam}, see also Remark~\ref{r:diforfullflag}.  Hence, after some simplification, 
\begin{align}\label{e:pflag-q-fomula}
  \mathbf q_{n_i}^{(n)}\left( {\left[ u^{\infty} \right]}_n \right) &=  \sqrt[{(n_{i+1}-n_i)(n_i- n_{i-1})}]{\frac{{\left( \prod\limits_{j=1}^{n_{i-1}}d_{j}^{(n)} \right)}^{n_{i+1}-n_i}{\left( \prod\limits_{j=1}^{n_{i+1}}d_{j} ^{(n)}\right)}^{n_i-n_{i-1}}}{{\left( \prod\limits_{j=1}^{n_{i}}d_{j}^{(n)}\right)}^{n_{i+1}-n_{i-1}}}}\\
  &= \frac{\sqrt[{(n_{i+1}-n_i)}]{{\left( \prod\limits_{j=n_i+1}^{n_{i+1}}d_{j}^{(n)} \right)}}}{\sqrt[{(n_{i}-n_{i-1})}]{{\left( \prod\limits_{j=n_{i-1}+1}^{n_{i}}d_{j}^{(n)} \right)}}}.\notag
\end{align}
Now by Theorem~\ref{t:fflag-d-lim}, \[
  \lim\limits_{n\to \infty} \sqrt[n]{d_{k}^{(n)}\left( {\left[ u^{\infty} \right]}_n \right)} = \alpha_{k},
\]
so that in the limit, using the fact that both $\alpha_{n_{i-1}+1} = \cdots = \alpha_{n_{i}}$ and  $\alpha_{n_{i}+1} = \cdots = \alpha_{n_{i+1}}$, we find that
\begin{equation}
  \lim\limits_{n\to \infty} \sqrt[n]{\mathbf q_{n_i}^{(n)}\left( {\left[ u^{\infty} \right]}_n \right)}  = \lim\limits_{n\to \infty}\frac{\sqrt[{(n_{i+1}-n_i)}]{{\left( \prod\limits_{j=n_i+1}^{n_{i+1}}\sqrt[n]{d_{j}^{(n)}} \right)}}}{\sqrt[{(n_{i}-n_{i-1})}]{{\left( \prod\limits_{j=n_{i-1}+1}^{n_{i}}\sqrt[n]{d_{j}^{(n)}} \right)}}}= \frac{\sqrt[{(n_{i+1}-n_i)}]{{\left( \prod\limits_{j=n_i+1}^{n_{i+1}}\alpha_{n_{i+1}} \right)}}}{\sqrt[{(n_{i}-n_{i-1})}]{{\left( \prod\limits_{j=n_{i-1}+1}^{n_{i}}\alpha_{n_{i}}\right)}}}= \frac{\alpha_{n_{i+1}}}{\alpha_{n_{i}}}. 
\end{equation}
On the other hand, 
\[
\lim_{n\to\infty}\sqrt[n]{{\mathbf q}^{(n)}_{n_{i}}(u^{(n)})}=\lim_{n\to\infty} \sqrt[n]{\mathbf q_{n_i}^{(n)}\left( {\left[ u^{\infty} \right]}_n \right)},\]
 as follows from the  formula \eqref{e:qnin} that defines $\mathbf q_{n_i}$  both on $[u^\infty]_n$ and $u^{(n)}$, and  the additional convergence condition on the $\Delta_{n_i}$. 

But since $ q^{(n)}_{n_{i}}(u^{(n)})=\mathbf q^{(n)}_{n_{i}}(u^{(n)})$ by \eqref{e:qviadelta}, we have 
\begin{align*}
\lim\limits_{n \to \infty}\sqrt[n]{q^{(n)}_{n_{i}}(u^{(n)})}  
=\frac{\alpha_{n_{i+1}}}{\alpha_{n_{i}}}. 
\end{align*}
For the second formula we have for each  $m_i\in \mathbf{m}$, that 

\begin{equation}
  q_{n-m_{i}}^{(n)} \left( {\left[ u^{\infty} \right]}_n \right) = \frac{\sqrt[{(m_i-m_{i-1})}]{{\left( \prod\limits_{j=n-m_i+1}^{n-m_{i-1}}d^{(n)}_{j} \right)}}}{\sqrt[{(m_{i+1}-m_{i})}]{{\left( \prod\limits_{j=n-m_{i+1}+1}^{n-m_{i}}d^{(n)}_{j} \right)}}}, 
\end{equation}
as a consequence of \eqref{e:pflag-q-fomula}. This, combined with Theorem~\ref{t:fflag-d-lim}, implies the limit formula 
\[
  \lim\limits_{n \to \infty}\sqrt[n]{\mathbf q^{(n)}_{n-m_{i}}([u^\infty]_{n})} = \frac{\beta_{m_{i+1}}}{\beta_{m_{i}}}.
\]
Now again $\sqrt[n]{\mathbf q^{(n)}_{n-m_{i}}(u^{(n)})}$ has the same limit as $ \sqrt[n]{\mathbf q_{n-m_i}^{(n)}\left( {\left[ u^{\infty} \right]}_n \right)}$ by the convergence condition on the $\Delta_{n-m_j}$ and $\sqrt[n]{\mathbf q^{(n)}_{n-m_{i}}(u^{(n)})}=\sqrt[n]{q^{(n)}_{n-m_{i}}(u^{(n)})}$, so that the  proposition follows.  
\end{proof}

    Note that already the statement of Proposition~\ref{p:pflag-q-lim-inf} requires both sequences $\mathbf n$ and $\mathbf m$ to be infinite. 
    
 For the more general case of partial flag varieties $\Fl_{(\mathbf n,\mathbf m)}$ where $\mathbf n$ and/or $\mathbf m$ may be finite, we make the following conjecture.  
\begin{conjecture}\label{c:pflag-q-lim}
Let $(\mathbf n,\mathbf m)$ be as in Definition~\ref{d:infiniteflagseqs}. Consider a pair of sequences $\boldsymbol{\alpha}_{\mathbf{n}}$ and $\boldsymbol{\beta}_{\mathbf{m}}$ as in Conjecture~\ref{c:pflagconjectures-alt}, and suppose we have a sequence $u^{(n)}\in X_{P_{(\mathbf{n},\mathbf{m})}^{(n)}}(\R_{> 0})$ converging as in Proposition~\ref{p:pflag-q-lim-inf} to $u^{\infty}$, with Schoenberg parameters given by $\boldsymbol{\alpha}_{\mathbf{n}}$ and $\boldsymbol{\beta}_{\mathbf{m}}$. If $\mathbf{n}$ is finite of length $k$, then 
\begin{equation}\label{e:qinfnfin}
  \lim\limits_{n\to \infty} \sqrt[n]{q_{n_{i}}^{(n)}(u^{(n)})} = \begin{cases}
      \frac{\alpha_{n_{i+1}}}{\alpha_{n_{i}}} & i\leq k-1, \\
      \frac{1}{\alpha_{n_{i}}} & i = k .\\
  \end{cases}
  \end{equation}
Similarly, if $\mathbf{m}$ is finite of length $\ell$, then
  \begin{equation}\label{e:qinfmfin}
      \lim\limits_{n\to \infty} \sqrt[n]{q_{n-m_{i}}^{(n)}(u^{(n)})} = \begin{cases}
          \frac{\beta_{m_{i+1}}}{\beta_{m_{i}}}, & i\leq \ell-1 \\
           \frac{1}{\beta_{m_{i}}} & i=\ell. 
      \end{cases} 
  \end{equation}
\end{conjecture}
\begin{remark}
This conjecture holds in the  Grassmannian setting, that is,  if either $\mathbf n=\{k\}$ and $\mathbf m=\emptyset$, or if  $\mathbf n=\emptyset$ and $\mathbf m=(m)$. Namely, in these two cases it follows from Propositions~\ref{prop:lim-q-param2} and ~\ref{prop:lim-q-param}, respectively, with formula \eqref{e:qinfnfin} recovering \eqref{e:qalpha}, and formula \eqref{e:qinfmfin} recovering \eqref{e:qtobeta}.
\end{remark}
\bibliographystyle{plain}
\bibliography{references}

\begin{thebibliography}{10}

\bibitem{AESW}
Michael Aissen, Albert Edrei, I.~J. Schoenberg, and Anne Whitney.
\newblock On the generating functions of totally positive sequences.
\newblock {\em Proc. Natl. Acad. Sci. USA}, 37:303--307, 1951.

\bibitem{ASW}
Michael Aissen, I.~J. Schoenberg, and Anne Whitney.
\newblock On the generating functions of totally positive sequences. {I}.
\newblock {\em J. Anal. Math.}, 2:93--103, 1952.

\bibitem{Arkani-Hamedetalbook}
Nima Arkani-Hamed, Jacob Bourjaily, Freddy Cachazo, Alexander Goncharov,
  Alexander Postnikov, and Jaroslav Trnka.
\newblock {\em Grassmannian geometry of scattering amplitudes}.
\newblock Cambridge University Press, Cambridge, 2016.

\bibitem{ArkaniHamed2013jha}
Nima Arkani-Hamed and Jaroslav Trnka.
\newblock {The Amplituhedron}.
\newblock {\em JHEP}, 10:030, 2014.

\bibitem{BaoHe}
Huanchen Bao and Xuhua He.
\newblock Product structure and regularity theorem for totally nonnegative flag
  varieties.
\newblock {\em Invent. Math.}, 237(1):1--47, 2024.

\bibitem{BerensteinZelevinsky}
Arkady Berenstein and Andrei Zelevinsky.
\newblock Total positivity in {S}chubert varieties.
\newblock {\em Comment. Math. Helv.}, 72(1):128--166, 1997.

\bibitem{BERTRAM1997289}
Aaron Bertram.
\newblock Quantum schubert calculus.
\newblock {\em Advances in Mathematics}, 128(2):289--305, 1997.

\bibitem{BCFF}
Aaron Bertram, Ionu\c~t{} Ciocan-Fontanine, and William Fulton.
\newblock Quantum multiplication of {S}chur polynomials.
\newblock {\em J. Algebra}, 219(2):728--746, 1999.

\bibitem{BDW}
Aaron Bertram, Georgios Daskalopoulos, and Richard Wentworth.
\newblock Gromov invariants for holomorphic maps from {R}iemann surfaces to
  {G}rassmannians.
\newblock {\em J. Amer. Math. Soc.}, 9(2):529--571, 1996.

\bibitem{BorodinOlshanskiBook}
Alexei Borodin and Grigori Olshanski.
\newblock {\em Representations of the infinite symmetric group}, volume 160 of
  {\em Camb. Stud. Adv. Math.}
\newblock Cambridge: Cambridge University Press, 2016.

\bibitem{BuchGrassmannianQCoh}
Anders~Skovsted Buch.
\newblock Quantum cohomology of {G}rassmannians.
\newblock {\em Compositio Math.}, 137(2):227--235, 2003.

\bibitem{BuchWang}
Anders~Skovsted Buch and Chengxi Wang.
\newblock Positivity determines the quantum cohomology of {G}rassmannians.
\newblock {\em Algebra Number Theory}, 15(6):1505--1521, 2021.

\bibitem{ChaputManivelPerrinIII}
P.~E. Chaput, L.~Manivel, and N.~Perrin.
\newblock Quantum cohomology of minuscule homogeneous spaces {III}.
  {S}emi-simplicity and consequences.
\newblock {\em Canad. J. Math.}, 62(6):1246--1263, 2010.

\bibitem{ChaputManivelPerrinII}
Pierre-Emmanuel Chaput, Laurent Manivel, and Nicolas Perrin.
\newblock Quantum cohomology of minuscule homogeneous spaces. {II}. {H}idden
  symmetries.
\newblock {\em Int. Math. Res. Not. IMRN}, (22):Art. ID rnm107, 29, 2007.

\bibitem{CookmeyerMilicevic}
Jonathan Cookmeyer and Elizabeth Mili\'cevi\'c.
\newblock Applying parabolic {P}eterson: affine algebras and the quantum
  cohomology of the {G}rassmannian.
\newblock {\em J. Comb.}, 10(1):129--162, 2019.

\bibitem{Edrei52}
Albert Edrei.
\newblock On the generating functions of totally positive sequences {II}.
\newblock {\em J. Analyse Math.}, 2:104--109, 1952.

\bibitem{FockGoncharovI}
Vladimir Fock and Alexander Goncharov.
\newblock Moduli spaces of local systems and higher {T}eichm\"uller theory.
\newblock {\em Publ. Math. Inst. Hautes \'Etudes Sci.}, (103):1--211, 2006.

\bibitem{FOMINICM2010}
Sergey Fomin.
\newblock {\em Total Positivity and Cluster Algebras}, pages 125--145.

\bibitem{FGP}
Sergey Fomin, Sergei Gelfand, and Alexander Postnikov.
\newblock Quantum {Schubert} polynomials.
\newblock {\em J. Am. Math. Soc.}, 10(3):565--596, 1997.

\bibitem{Grochenig2023}
Karlheinz Gr{\"o}chenig.
\newblock {\em Schoenberg's Theory of Totally Positive Functions and the
  Riemann Zeta Function}, pages 193--210.
\newblock Springer International Publishing, Cham, 2023.

\bibitem{hengelbrock}
Harald Hengelbrock.
\newblock An involution on the quantum cohomology ring of the {G}rassmannian.
\newblock https://arxiv.org/abs/math/0205260, 2002.

\bibitem{Karlin}
S.~Karlin.
\newblock Total positivity. {Volume} {I}.
\newblock {Stanford} {University} {Press}. {XIV}, 576 p., 1968.

\bibitem{Katkova}
Olga~M. Katkova.
\newblock Multiple positivity and the {R}iemann zeta-function.
\newblock {\em Comput. Methods Funct. Theory}, 7(1):13--31, 2007.

\bibitem{KodamaWilliams}
Yuji Kodama and Lauren Williams.
\newblock {KP} solitons and total positivity for the {Grassmannian}.
\newblock {\em Invent. Math.}, 198(3):637--699, 2014.

\bibitem{Kostant:qcoh}
Bertram Kostant.
\newblock Flag manifold quantum cohomology, the {T}oda lattice, and the
  representation with highest weight {$\rho$}.
\newblock {\em Selecta Math. (N.S.)}, 2(1):43--91, 1996.

\bibitem{LamPP}
Thomas Lam and Pavlo Pylyavskyy.
\newblock Total positivity in loop groups. {I}: {Whirls} and curls.
\newblock {\em Adv. Math.}, 230(3):1222--1271, 2012.

\bibitem{LamRietsch}
Thomas Lam and Konstanze Rietsch.
\newblock Total positivity, {Schubert} positivity, and geometric {Satake}.
\newblock {\em J. Algebra}, 460:284--319, 2016.

\bibitem{LRY:Schubert}
Changzheng Li, Konstanze Rietsch, and Mingzhi Yang.
\newblock A {P}eterson program for general {S}chubert varieties and mirror
  symmetry.
\newblock 2026.

\bibitem{Lusztig94}
George Lusztig.
\newblock Total positivity in reductive groups.
\newblock In {\em Lie theory and geometry}, volume 123 of {\em Progr. Math.},
  pages 531--568. Birkh\"auser Boston, Boston, MA, 1994.

\bibitem{LUSZTIG1997}
George Lusztig.
\newblock Total positivity and canonical bases.
\newblock In {\em Algebraic groups and {L}ie groups}, volume~9 of {\em Austral.
  Math. Soc. Lect. Ser.}, pages 281--295. Cambridge Univ. Press, Cambridge,
  1997.

\bibitem{Lusztig1998TotalPI}
George Lusztig.
\newblock Total positivity in partial flag manifolds.
\newblock {\em Representation Theory of The American Mathematical Society},
  2:70--78, 1998.

\bibitem{Lusztig2019TotalPI}
George Lusztig.
\newblock Total positivity in reductive groups, ii.
\newblock {\em Bulletin of the Institute of Mathematics Academia Sinica NEW
  SERIES}, 2019.

\bibitem{LUSZTIG2023}
George Lusztig.
\newblock On the totally positive grassmannian.
\newblock {\em Bull. Math. Soc. Sci. Math. Roumanie (N.S.)},
  66(114)(4):455--458, 2023.

\bibitem{peterson}
D.~Peterson.
\newblock {Quantum cohomology of $G/P$}.
\newblock {Lecture Course, MIT, Spring Term}, 1997.

\bibitem{PostnikovAffine}
Alexander Postnikov.
\newblock Affine approach to quantum {S}chubert calculus.
\newblock {\em Duke Math. J.}, 128(3):473--509, 2005.

\bibitem{postnikovgrassmannian}
Alexander Postnikov.
\newblock Total positivity, grassmannians, and networks, 2006.

\bibitem{rietsch2001grassmannians}
Konstanze Rietsch.
\newblock {Quantum cohomology rings of Grassmannians and total positivity}.
\newblock {\em Duke Mathematical Journal}, 110(3):523 -- 553, 2001.

\bibitem{rietsch2001flagvarieties}
Konstanze Rietsch.
\newblock Totally positive {Toeplitz} matrices and quantum cohomology of
  partial flag varieties.
\newblock {\em J. Am. Math. Soc.}, 16(2):363--392, 2003.

\bibitem{rietschNagoya}
Konstanze Rietsch.
\newblock A mirror construction for the totally nonnegative part of the
  {P}eterson variety.
\newblock {\em Nagoya Math. J.}, 183:105--142, 2006.

\bibitem{rietsch2001err}
Konstanze Rietsch.
\newblock Errata to: ``{T}otally positive {T}oeplitz matrices and quantum
  cohomology of partial flag varieties'' [{J}. {A}mer. {M}ath. {S}oc. {\bf 16}
  (2003), no. 2, 363--392; mr1949164].
\newblock {\em J. Amer. Math. Soc.}, 21(2):611--614, 2008.

\bibitem{rietsch:mirror}
Konstanze Rietsch.
\newblock A mirror symmetric construction of {$qH^\ast_T(G/P)_{(q)}$}.
\newblock {\em Adv. Math.}, 217(6):2401--2442, 2008.

\bibitem{rietsch2025totallypositivetoeplitzmatrices}
Konstanze Rietsch.
\newblock Totally positive {T}oeplitz matrices: classical and modern.
\newblock https://arxiv.org/abs/2509.25163, 2025.

\bibitem{rietsch2025tropicaltoeplitzmatricesparametrisations}
Konstanze Rietsch.
\newblock Tropical toeplitz matrices and parametrisations.
\newblock arXiv:2509.06944[math.rt], 2025.

\bibitem{Schoenberg:48}
I.~J. Schoenberg.
\newblock Some analytical aspects of the problem of smoothing.
\newblock In {\em Studies and {E}ssays {P}resented to {R}. {C}ourant on his
  60th {B}irthday, {J}anuary 8, 1948}, pages 351--370. 1948.

\bibitem{SiebertTian}
Bernd Siebert and Gang Tian.
\newblock On quantum cohomology rings of {F}ano manifolds and a formula of
  {V}afa and {I}ntriligator.
\newblock {\em Asian J. Math.}, 1(4):679--695, 1997.

\bibitem{Springer:book}
T.~A. Springer.
\newblock {\em Linear algebraic groups}.
\newblock Modern Birkh\"{a}user Classics. Birkh\"{a}user Boston, Inc., Boston,
  MA, second edition, 2009.

\bibitem{Thoma}
E.~Thoma.
\newblock Die unzerlegbaren, positive–definiten {K}lassenfunktionen der
  abz\"ahlbar un-endlichen, symmetrischen {G}ruppe.
\newblock {\em Math. Zeitschr.}, 85:40--6, 1964.

\bibitem{KEROVVERSHIK1981}
A.~M. Vershik and S.~V. Kerov.
\newblock Asymptotic theory of characters of the symmetric group.
\newblock {\em Functional Analysis and Its Applications}, 15(4):246--255,
  October 1981.

\bibitem{WittenGrassmannian}
Edward Witten.
\newblock The {V}erlinde algebra and the cohomology of the {G}rassmannian.
\newblock In {\em Geometry, topology, \& physics}, volume~IV of {\em Conf.
  Proc. Lecture Notes Geom. Topology}, pages 357--422. Int. Press, Cambridge,
  MA, 1995.

\end{thebibliography}

\end{document}